\documentclass[preprint,12pt]{elsarticle}

\usepackage{amssymb}
\usepackage{amsmath}
\usepackage{amsthm}

\usepackage{xcolor}
\usepackage{hyperref}
\usepackage{natbib}
\usepackage{tikz}
\usetikzlibrary{arrows}
\newtheorem{defn}{Definition}
\newtheorem{thm}{Theorem}

\newtheorem{cor}{Corollary}

\newdefinition{rmk}{Remark}

\journal{}

\begin{document}

\begin{frontmatter}



\title{Higher-Order Quadripolar Argumentation Framework and Encoded Semantics}


\author[a]{Shuai Tang\corref{cor1}} 
\ead{TangShuaiMath@outlook.com}

\cortext[cor1]{Corresponding author}

\begin{abstract}
	This paper introduces the Higher-Order Quadripolar Argumentation Framework (HQAF), which extends bipolar frameworks by allowing attacks and three types of supports---necessary, deductive, and evidential---to interact at higher order. In HQAF, each interaction may have as its source and target an argument, an attack, or a support. We provide a 3-valued semantics via adjacent complete labellings and equational systems, and define encoded semantics directly in propositional logics, mainly Łukasiewicz three-valued logic for the 3-valued case and fuzzy logics based on continuous t-norms (Gödel, Product, Łukasiewicz) for the fuzzy case. Model equivalence between HQAF and its normal encoding is established. The fuzzy equational semantics is shown to be a specific formalism of the fuzzy encoded semantics, and ternarization connects fuzzy models to 3-valued complete labellings. The study is confined to non-set HQAF; collective interactions are reserved for future work.
\end{abstract}



\end{frontmatter}



\section{Introduction}
\label{sec:intro}

Argumentation has become an essential paradigm in Artificial Intelligence for reasoning from incomplete and contradictory information, and for modelling the exchange of arguments between agents \cite{Dung1995,Rahwan2009,Bench-Capon2007}. Dung's abstract argumentation framework (AF) \cite{Dung1995} consists of a set of arguments and a binary attack relation, and determines acceptable sets of arguments called extensions. Over the last decades, AFs have been extended along two main directions: adding positive interactions expressed by a support relation, leading to bipolar argumentation frameworks (BAF) \cite{Cayrol2005,Amgoud2008,Cayrol2013}; and allowing higher-order interactions, i.e., attacks or supports whose targets are other attacks or supports \cite{Barringer2005,Modgil2009,Baroni2011,Cohen2015}.

The notion of support in BAF is left abstract, and several specialized interpretations have been proposed. Necessary support \cite{Nouioua2010,Nouioua2011,Nouioua2013} captures the intuition that the acceptance of the target requires the acceptance of each support source, or equivalently, the existing non-acceptance of the support source implies the non-acceptance of the target. Deductive support \cite{Boella2010} captures the dual intuition that the acceptance of the source implies the acceptance of the target. As observed by Cayrol and Lagasquie-Schiex \cite{Cayrol2013}, necessary and deductive supports are dual notions in the following sense: \(A\) necessarily supports \(B\) if and only if \(B\) deductively supports \(A\). Evidential support \cite{Oren2008,Oren2010} distinguishes prima-facie arguments from standard ones, where standard arguments must be supported by a chain rooted in prima-facie arguments. Comparative studies and translations between these interpretations are given in \cite{Cayrol2013,Polberg2014}. However, existing bipolar frameworks typically consider at most one support interpretation at a time.

A second direction extends AFs with higher-order interactions. AFRA \cite{Baroni2011} and RAF \cite{Cayrol2017} allow attacks to target attacks. More recently, ASAF \cite{Cohen2015,Gottifredi2018} combines recursive attacks with necessary support; RAFN \cite{Cayrol2018a} handles higher-order necessary supports; and REBAF \cite{Cayrol2018b} handles higher-order evidential supports. Logical encodings for RAFN and REBAF are developed in \cite{Cayrol2020,Lagasquie-Schiex2021a,Lagasquie-Schiex2021b,Lagasquie-Schiex2023}. These frameworks are bipolar: they pair attacks with one kind of support, and often restrict support sources to sets of arguments. Gabbay \cite{Gabbay2016} introduces multipolar and tripolar argumentation networks, but the framework is not higher-order and lacks encoded semantics. Tang \cite{Tang2025} proposes HAFS, a higher-order framework with necessary supports, and encodes it into propositional logic. A related evidential higher-order set framework is studied in \cite{Tang2026}. Thus, a uniform framework that simultaneously accommodates attacks, necessary supports, deductive supports, and evidential supports—allowing them to act as sources and targets at higher order—is still missing.

To fill this gap, this paper proposes a \emph{Higher-Order Quadripolar Argumentation Framework} (HQAF). HQAF extends existing bipolar and multipolar frameworks by incorporating four kinds of interactions: attacks, necessary supports, deductive supports, and evidential supports. In HQAF, not only arguments but also attacks and supports can act as sources and targets of interactions, yielding a uniform treatment of higher-order interactions. We define a suite of semantics for HQAF, including adjacent complete labelling semantics (a 3-valued semantics), 3-valued equational semantics, and encoded semantics. We directly define the​ normal encoded semantics by interpreting the normal encoded formula of HQAF in propositional logic systems: the 3-valued semantics is encoded into three-valued propositional logic (typically Łukasiewicz's three-valued propositional logic), and the fuzzy semantics is obtained by encoding HQAF into fuzzy propositional logics such as Gödel, Product, and Łukasiewicz fuzzy logics. We prove model equivalence between fuzzy equational semantics and fuzzy normal encoded semantics, establishing the logical foundation of HQAF semantics. In particular, the three key fuzzy equational semantics are not introduced separately; they arise directly as special cases of the fuzzy encoded semantics. Additionally, we investigate the relationships between 3-valued complete semantics and fuzzy encoded semantics, showing that models of fuzzy encoded semantics can be transformed into complete semantics models via ternarization, and vice versa for specific t-norms.

The main contributions of this paper are summarized as follows.
\begin{itemize}
	\item We propose HQAF as a syntactic extension of existing bipolar and multipolar frameworks, enabling mutual interactions among attacks and three types of supports (necessary, deductive, and evidential), and allowing attacks and supports to serve as both targets and sources.
	\item We define adjacent complete labelling semantics and 3-valued equational semantics for HQAF, and directly define encoded semantics in propositional logics, where the fuzzy equational semantics is obtained as a derivative of the fuzzy normal encoded semantics.
	\item We investigate the relationships between 3-valued complete semantics and fuzzy encoded semantics via ternarization.
\end{itemize}

The remainder of the paper is organized as follows. Section~\ref{sec:prelim} presents background on propositional logic systems and reviews necessary-based and evidential-based higher-order argumentation frameworks. Section~\ref{sec:hqaf-syntax} introduces the syntax and basic semantics of HQAF. Section~\ref{sec:encoding} develops the encoded semantics for HQAF, including discrete normal encoding, continuous fuzzy operator-based equational semantics, and continuous fuzzy normal encoding, and establishes their equivalence. Section~\ref{sec:related} discusses related work, and Section~\ref{sec:conclusion} concludes with future directions.

\section{Preliminaries}\label{sec:prelim}
In \cite{Tang2025} the author proposes the higher-order argumentation framework with supports (HAFS). In \cite{Tang2026} the author introduces the evidential-based higher-order set argumentation framework (EHSAF).
As presented in the two papers, the encoded semantics are obtained by encoding a framework into a chosen $\mathcal{PL}$ and solving the models of the encoded formula.
In this section, we first review some basic knowledge of the $\mathcal{PL}$ and then review the syntax and encoded semantics for the two kinds of frameworks, where we rename the HAFS as the necessary-based higher-order argumentation framework (NHAF) and restrict the EHSAF as the non-set style named the evidential-based higher-order argumentation framework (EHAF). 
\subsection{Propositional Logic Systems}

We presuppose basic familiarity with propositional logic ($\mathcal{PL}$) systems; readers seeking a more thorough exposition may consult \cite{Klir1995,Hajek1998,Klement2000,Bergmann2008,Belohlavek2017}. In this subsection, we fix notation and recap the core definitions needed for subsequent developments.

At the meta-level, we use $A \Longleftrightarrow B$ to abbreviate the statement ``$A$ if and only if $B$'', and $B \Longrightarrow A$ to abbreviate ``if $B$ then $A$'', where $A$ and $B$ range over arbitrary meta-language sentences.

A formal language $\mathcal{L}$ for a propositional logic is specified by three components:

\begin{enumerate}
	\item \textbf{Alphabet.} The alphabet of $\mathcal{L}$ consists of:
	\begin{itemize}
		\item A countably infinite set $\mathit{Var} = \{p_1, p_2, p_3, \dots\}$ of propositional variables;
		\item A collection of logical connectives $\{\neg, \wedge, \vee, \rightarrow, \leftrightarrow\}$, where $\neg$ is a unary connective (negation), and $\wedge, \vee, \rightarrow, \leftrightarrow$ are binary connectives (conjunction, disjunction, implication, and equivalence, respectively);
		\item Auxiliary punctuation: the left and right parentheses ``('' and ``)''.
	\end{itemize}
	
	\item \textbf{Syntax (Formation Rules).} The set $F_{\mathcal{PL}}$ of \emph{well-formed formulas} (wffs) is the smallest set generated by the following inductive rules:
	\begin{itemize}
		\item Every propositional variable $p \in \mathit{Var}$ is a wff, i.e.\ $p \in F_{\mathcal{PL}}$;
		\item If $\varphi \in F_{\mathcal{PL}}$, then its negation $(\neg \varphi)$ is also a wff;
		\item If $\varphi, \psi \in F_{\mathcal{PL}}$ and $\circ \in \{\wedge, \vee, \rightarrow, \leftrightarrow\}$, then the compound formula $(\varphi \circ \psi)$ is a wff.
	\end{itemize}
	
	\item \textbf{Notational Conventions.} To simplify presentation, we adopt the following standard conventions:
	\begin{itemize}
		\item Outermost parentheses are omitted when no ambiguity arises;
		\item Connectives are ordered by precedence (from highest to lowest): $\neg$, $\wedge$, $\vee$, $\rightarrow$, $\leftrightarrow$;
		\item The propositional constants $\top$ (tautology) and $\bot$ (contradiction) are taken as primitive symbols.
	\end{itemize}
\end{enumerate}

In this work, we exclusively consider propositional logics whose truth-value domains are subsets of the unit interval $[0,1]$, together with their corresponding numerical semantics. Under this approach, each connective receives a fixed interpretation as a truth function (operator) on the domain. We use the notation $\|\cdot\|$ uniformly to denote both an assignment (model) defined on $\mathit{Var}$ and the unique extension of this assignment to evaluate all well-formed formulas. For every such $\mathcal{PL}$ and every assignment $\|\cdot\|$, we require the constant conditions $\|\bot\| = 0$ and $\|\top\| = 1$, and we define equivalence by $a \leftrightarrow b := (a \rightarrow b) \wedge (b \rightarrow a)$ for arbitrary formulas $a$ and $b$.

\subsubsection{Łukasiewicz Three-Valued Logic}

We now recall the three-valued propositional logic $\mathcal{PL}_3^L$ originally due to Łukasiewicz \cite{Lukasiewicz1970}. The connectives $\neg$, $\wedge$ and $\vee$ are interpreted pointwise as follows:
\begin{itemize}
	\item $\|\neg a\| = 1 - \|a\|$,
	\item $\|a \wedge b\| = \min\left\{\|a\|, \|b\|\right\}$,
	\item $\|a \vee b\| = \max\left\{\|a\|, \|b\|\right\}$,
\end{itemize}
where $a$ and $b$ range over arbitrary propositional formulas.

The implication connective $\rightarrow$ is interpreted via the three-valued residuum operator $\Rightarrow$, so that $\|a \rightarrow b\| = \|a\| \Rightarrow \|b\|$. The truth-degree table for $\Rightarrow$ is given in Table~\ref{Rightarrow}.

\begin{table}[htbp]
	\centering
	\renewcommand{\arraystretch}{1.2}
	\begin{tabular}{c | c c c}
		$\Rightarrow$ & $0$ & $\frac{1}{2}$ & $1$ \\ 
		\hline
		$0$ & $1$ & $1$ & $1$ \\
		$\frac{1}{2}$ & $\frac{1}{2}$ & $1$ & $1$ \\
		$1$ & $0$ & $\frac{1}{2}$ & $1$
	\end{tabular}
	\caption{Truth‑degree table of $\Rightarrow$ in $\mathcal{PL}_3^L$.}
	\label{Rightarrow}
\end{table}

Given the definition of equivalence, we immediately obtain:
\[
\|a \leftrightarrow b\| 
= \min\left\{\|a \rightarrow b\|, \|b \rightarrow a\|\right\} 
= \min\left\{\|a\| \Rightarrow \|b\|, \|b\| \Rightarrow \|a\|\right\}.
\]
Hence the equivalence connective is interpreted by the operator $\Leftrightarrow$, whose truth-degree table is displayed in Table~\ref{Leftrightarrow}.

\begin{table}[htbp]
	\centering
	\renewcommand{\arraystretch}{1.2}
	\begin{tabular}{c | c c c}
		$\Leftrightarrow$ & $0$ & $\frac{1}{2}$ & $1$ \\ 
		\hline
		$0$ & $1$ & $\frac{1}{2}$ & $0$ \\
		$\frac{1}{2}$ & $\frac{1}{2}$ & $1$ & $\frac{1}{2}$ \\
		$1$ & $0$ & $\frac{1}{2}$ & $1$
	\end{tabular}
	\caption{Truth‑degree table of $\Leftrightarrow$ in $\mathcal{PL}_3^L$.}
	\label{Leftrightarrow}
\end{table}

\subsubsection{Continuous Fuzzy Logics and T-Norms}

Next, we review the core operators of fuzzy propositional logics over the continuous domain $[0,1]$, referred to as $\mathcal{PL}_{[0,1]}$ systems \cite{Klir1995,Hajek1998,Klement2000}. In these logics, $\neg$ is interpreted as a negation function, $\wedge$ as a triangular norm (t-norm), and $\rightarrow$ as the residuated implication induced by the t-norm.

A negation function $N$ is called the \emph{standard negation} if $N(m) = 1 - m$ for all $m \in [0,1]$. The three most prominent continuous t-norms, together with their corresponding $R$-implications, are listed below (for all $m, n \in [0,1]$):
\begin{itemize}
	\item \emph{Gödel t-norm}: $T_G(m, n) = \min\{m, n\}$, with residuum
	\[
	I_G(m, n) = 
	\begin{cases}
		1 & m \leq n, \\
		n & m > n;
	\end{cases}
	\]
	\item \emph{Łukasiewicz t-norm}: $T_L(m, n) = \max\{0, m + n - 1\}$, with residuum $I_L(m, n) = \min\{1 - m + n, 1\}$;
	\item \emph{Product t-norm}: $T_P(m, n) = m \cdot n$, with residuum
	\[
	I_P(m, n) =
	\begin{cases}
		1 & m \leq n, \\
		\frac{n}{m} & m > n.
	\end{cases}
	\]
\end{itemize}

A fundamental property of any residuated implication $I$ is that $I(m, n) = I(n, m) = 1$ if and only if $m = n$ for all $m, n \in [0,1]$. As a consequence, $\|a \leftrightarrow b\| = 1$ holds precisely when $\|a\| = \|b\|$, since both directions of the implication must evaluate to $1$.

In the remainder of the paper, we use the following notation for the three canonical fuzzy logics, all equipped with the standard negation \cite{Esteva2000}:
\begin{itemize}
	\item $\mathcal{PL}_{[0,1]}^G$: the logic based on the Gödel t-norm $T_G$ and its residuum $I_G$;
	\item $\mathcal{PL}_{[0,1]}^P$: the logic based on the Product t-norm $T_P$ and its residuum $I_P$;
	\item $\mathcal{PL}_{[0,1]}^L$: the logic based on the Łukasiewicz t-norm $T_L$ and its residuum $I_L$.
\end{itemize}
\subsection{The syntax and encoded semantics of NHAF}
\begin{defn}[\cite{Tang2025}]\label{pNHAF}
	A preparatory necessary-based higher-order argumentation framework ($pre\text{-}NHAF$) is a 4-tuple $(\mathbb{A}^N, \mathbb{R}_k^N, \mathbb{R}_n^N, \mathbb{U}^N)$, where $\mathbb{A}^N$ is a set of arguments, $\mathbb{R}_k^N$ is named the attack relation, $\mathbb{R}_n^N$ is named the necessary support relation, $\mathbb{R}_k^N\cup \mathbb{R}_n^N\subseteq (\mathbb{A}^N \cup \mathbb{R}_k^N \cup\mathbb{R}_n^N) \times (\mathbb{A}^N \cup \mathbb{R}_k^N \cup\mathbb{R}_n^N)$, the finite universal set for the pre-NHAF is $\mathbb{U}^N=\mathbb{A}^N \cup \mathbb{R}_k^N\cup\mathbb{R}_n^N$.
\end{defn}

\begin{defn}[\cite{Tang2025}]\label{NHAF}
	Given a $pre\text{-}NHAF = (\mathbb{A}^N, \mathbb{R}_k^N, \mathbb{R}_n^N, \mathbb{U}^N)$, a \emph{necessary-based higher-order argumentation framework} (NHAF)  is a 4-tuple $(\mathbf{A}^N, \mathbf{R_k}^N,$ $\mathbf{R_n}^N, \mathbf{U}^N)$, where the argument set $\mathbf{A}^N=\mathbb{A}^N\cup\{\bot\}\cup\{\top\}$, the attack relation $\mathbf{R_k}^N=\mathbb{R}_k^N\cup \{(\bot, \beta)\mid \beta\in \mathbb{U}^N \text{ and } \nexists \alpha\in \mathbb{U}^N: (\alpha, \beta)\in \mathbb{R}_k^N\}$, the necessary support relation $\mathbf{R_n}^N=\mathbb{R}_n^N\cup \{(\top, \beta)\mid \beta\in \mathbb{U}^N \text{ and } \nexists \alpha\in \mathbb{U}^N: (\alpha, \beta)\in \mathbb{R}_n^N\}$, the universal set for the NHAF is $\mathbf{U}^N=\mathbf{A}^N\cup \mathbf{R_k}^N\cup \mathbf{R_n}^N$.
\end{defn}
In any NHAF system or any propositional logic system, for any assignment $\|\cdot\|: \mathbf{U}^N\to [0,1]$, we always let the evaluation of constants satisfy that $\|\bot\|=0$, $\|\top\|=1$, $\|x\|=1$ for each $x\in\mathbf{R_k}^N\setminus\mathbb{R}_k^N$, and $\|y\|=1$ for each $y\in\mathbf{R_n}^N\setminus\mathbb{R}_n^N$.

Denote the set of all NHAF as $\mathcal{NH}$. Let the set of all elements in all NHAF are the subset of the set of all propositional variables and constants in any $\mathcal{PL}$. A well studied encoded formula for NHAF is the normal encoded formula as follows.
\begin{defn}[\cite{Tang2025}]
	The normal encoding of NHAF w.r.t a $\mathcal{PL}$ is a function $ec_{NH}: \mathcal{NH} \to F_{\mathcal{PL}}$ such that for a given $NHAF=(\mathbf{A}^N, \mathbf{R_k}^N, \mathbf{R_n}^N, \mathbf{U}^N)$ we have that
	\begin{equation*}
		ec_{NH}(NHAF)=\bigwedge_{a\in \mathbb{U}^N}(a\leftrightarrow(\bigwedge_{(b, a)\in\mathbf{R_k}^N}\neg(r_{b}^{a}\wedge b))\wedge(\bigwedge_{(c, a)\in \mathbf{R_n}^N}\neg(t_{c}^{a}\wedge \neg c))),
	\end{equation*}
	where $r_b^a$ denote any attack $(b,a)\in\mathbf{R_k}^N$ and $t_c^a$ denote any necessary support $(c,a)\in\mathbf{R_n}^N$.  
\end{defn}

\subsection{The syntax and encoded semantics of EHAF}
\begin{defn}[\cite{Tang2026}]\label{pEHAF}
	A preparatory evidential-based higher-order argumentation framework ($pre\text{-}EHAF$) is a 5-tuple $(\mathbb{A}^E, \mathbb{R}_k^E, \mathbb{R}_e^E, \mathbb{U}^E, \mathbb{P}^E)$, where $\mathbb{A}^E$ is a set of arguments, $\mathbb{R}_k^E$ is the attack relation, $\mathbb{R}_e^E$ is the evidential support relation, $\mathbb{R}_k^E\cup \mathbb{R}_e^E \subseteq (\mathbb{A}^E \cup \mathbb{R}_k^E \cup \mathbb{R}_e^E) \times (\mathbb{A}^E \cup \mathbb{R}_k^E\cup \mathbb{R}_e^E)$, the finite universal set of the pre-EHAF is $\mathbb{U}^E=\mathbb{A}^E \cup \mathbb{R}_k^E\cup \mathbb{R}_e^E$ and $\mathbb{P}^E\subseteq \mathbb{U}^E$ is called the set of prima-facie elements.
\end{defn}
\begin{defn}[\cite{Tang2026}]\label{EHAF}
	Given a $pre\text{-}EHAF = (\mathbb{A}^E, \mathbb{R}_k^E, \mathbb{R}_e^E, \mathbb{U}^E, \mathbb{P}^E)$, a \emph{evidential-based higher-order argumentation framework} (EHAF)  is a 5-tuple $(\mathbf{A}^E, \mathbf{R_k}^E,$ $\mathbf{R_e}^E, \mathbf{U}^E, \mathbf{P}^E)$, where the argument set $\mathbf{A}^E=\mathbb{A}^E\cup\{\bot\}\cup\{\top\}$, the attack relation $\mathbf{R_k}^E=\mathbb{R}_k^E\cup \{(\bot, \beta)\mid \beta\in \mathbb{U}^E \text{ and } \nexists \alpha\in \mathbb{U}^E: (\alpha, \beta)\in \mathbb{R}_k^E\}$, the evidential support relation $\mathbf{R_e}^E=\mathbb{R}_e^E\cup \{(\bot, \beta)\mid \beta\in \mathbb{U}^E\setminus \mathbb{P}^E \text{ and } \nexists \alpha\in \mathbb{U}^E: (\alpha, \beta)\in \mathbb{R}_e^E\}\cup \{(\top, \beta)\mid \beta\in \mathbb{P}^E \}$, the universal set $\mathbf{U}^E=\mathbf{A}^E\cup \mathbf{R_k}^E\cup \mathbf{R_e}^E$ and the prima-facie set $\mathbf{P}^E=\mathbb{P}^E\cup\{\top\}\cup(\mathbf{R_k}^E\setminus \mathbb{R}_k^E)\cup(\mathbf{R_e}^E\setminus \mathbb{R}_e^E)$.
\end{defn}

In any EHAF system or any propositional logic system, for any assignment $\|\cdot\|: \mathbf{U}^E\to [0,1]$, we always let the evaluation of constants satisfy that $\|\bot\|=0$, $\|\top\|=1$, $\|x\|=1$ for each $x\in\mathbf{R_k}^E\setminus\mathbb{R}_k^E$, and $\|y\|=1$ for each $y\in\mathbf{R_e}^E\setminus\mathbb{R}_e^E$.

Denote the set of all EHAF as $\mathcal{EH}$.  Let the set of all elements in all EHAF are the subset of the set of all propositional variables and constants in any $\mathcal{PL}$. A well studied encoded formula for EHAF is the normal encoded formula as follows.
\begin{defn}[\cite{Tang2026}]
	The normal encoding of EHAF w.r.t a $\mathcal{PL}$ is a function $ec_{EH}: \mathcal{EH} \to F_{\mathcal{PL}}$ such that for a given $EHAF=(\mathbf{A}^E, \mathbf{R_k}^E, \mathbf{R_e}^E, \mathbf{U}^E, \mathbf{P}^E)$ we have that
	\begin{equation*}
		ec_{EH}(EHAF)=\bigwedge_{a\in \mathbb{U}^E}(a\leftrightarrow(\bigwedge_{(b, a)\in\mathbf{R_k}^E}\neg(u_{b}^{a}\wedge b))\wedge(\bigvee_{(c, a)\in \mathbf{R_e}^E}(v_{c}^{a}\wedge  c))),
	\end{equation*}
	where $u_b^a$ denote any attack $(b,a)\in\mathbf{R_k}^E$ and $v_c^a$ denote any evidential support $(c,a)\in\mathbf{R_e}^E$.  
\end{defn}

\section{The Syntax and Basic Semantic Concepts of HQAF}\label{sec:hqaf-syntax}

\subsection{Syntax of HQAF}
	\begin{defn}\label{pHQAF}
	A preparatory higher-order quadripolar argumentation framework ($pre\text{-}HQAF$) is a 7-tuple $(\mathbb{A}, \mathbb{R}_k, \mathbb{R}_n, \mathbb{R}_d, \mathbb{R}_e, \mathbb{U}, \mathbb{P})$, where $\mathbb{A}$ is a set of arguments, $\mathbb{R}_k$ is named the attack relation, $\mathbb{R}_n$ is named the necessary support relation, $\mathbb{R}_d$ is named the deductive support relation, $\mathbb{R}_e$ is named the evidential support relation, $\mathbb{R}_k\cup \mathbb{R}_n\cup \mathbb{R}_d\cup \mathbb{R}_e \subseteq (\mathbb{A} \cup \mathbb{R}_k \cup\mathbb{R}_n\cup \mathbb{R}_d\cup \mathbb{R}_e) \times (\mathbb{A} \cup \mathbb{R}_k\cup \mathbb{R}_n\cup \mathbb{R}_d\cup\mathbb{R}_e)$, the finite universal set $\mathbb{U}=\mathbb{A} \cup \mathbb{R}_k\cup\mathbb{R}_n\cup \mathbb{R}_d\cup \mathbb{R}_e$ and $\mathbb{P}\subseteq \mathbb{U}$ called the set of prima-facie elements.
\end{defn}

\begin{defn}\label{HQAF}
	Given a $pre\text{-}HQAF = (\mathbb{A}, \mathbb{R}_k, \mathbb{R}_n, \mathbb{R}_d, \mathbb{R}_e, \mathbb{U}, \mathbb{P})$, a \emph{higher-order quadripolar argumentation framework} (HQAF)  is a 7-tuple $HQAF = (\mathbf{A}, \mathbf{R_k}, \mathbf{R_n}, \mathbf{R_d}, \mathbf{R_e}, \mathbf{U}, \mathbf{P})$, where the argument set $\mathbf{A}=\mathbb{A}\cup\{\bot\}\cup\{\top\}$, the attack relation $\mathbf{R_k}=\mathbb{R}_k\cup \{(\bot, \beta)\mid \beta\in \mathbb{U} \text{ and } \nexists \alpha\in \mathbb{U}: (\alpha, \beta)\in \mathbb{R}_k\}$, the necessary support relation $\mathbf{R_n}=\mathbb{R}_n\cup \{(\top, \beta)\mid \beta\in \mathbb{U} \text{ and } \nexists \alpha\in \mathbb{U}: (\alpha, \beta)\in \mathbb{R}_n\}$, the deductive support relation $\mathbf{R_d}=\mathbb{R}_d\cup \{(\beta, \top)\mid \beta\in \mathbb{U} \text{ and } \nexists \alpha\in \mathbb{U}: (\beta, \alpha)\in \mathbb{R}_d\}$, the evidential support relation $\mathbf{R_e}=\mathbb{R}_e\cup \{(\{\bot\}, \beta)\mid \beta\in \mathbb{U}\setminus \mathbb{P} \text{ and } \nexists \alpha\in \mathbb{U}: (\alpha, \beta)\in \mathbb{R}_e\}\cup \{(\{\top\}, \beta)\mid \beta\in \mathbb{P} \}$, the universal set $\mathbf{U}=\mathbf{A}\cup \mathbf{R_k}\cup \mathbf{R_n}\cup \mathbf{R_d}\cup \mathbf{R_e}$ and the prima-facie set $\mathbf{P}=\mathbb{P}\cup\{\top\}\cup(\mathbf{R_k}\setminus \mathbb{R}_k)\cup(\mathbf{R_n}\setminus \mathbb{R}_n)\cup(\mathbf{R_d}\setminus \mathbb{R}_d)\cup(\mathbf{R_e}\setminus \mathbb{R}_e)$.
\end{defn}
We call $\bot$, $\top$, or any $\alpha \in (\mathbf{R}_k \cup \mathbf{R}_n \cup \mathbf{R}_d \cup \mathbf{R}_e) \setminus (\mathbb{R}_k \cup \mathbb{R}_n \cup \mathbb{R}_d \cup \mathbb{R}_e)$ an auxiliary element.
To distinguish attacks and supports in the follows, we denote an attack $(\alpha, \beta)\in\mathbf{R_k}$ as $\mathbf{k}_\alpha^\beta$, a necessary support $(\gamma, \delta)\in\mathbf{R_n}$ as $\mathbf{n}_\gamma^\delta$, a deductive support $(\epsilon, \zeta)\in\mathbf{R_d}$ as $\mathbf{d}_\epsilon^\zeta$, and an evidential support $(\eta, \theta)\in\mathbf{R_e}$ as $\mathbf{e}_\eta^\theta$. Here, $\alpha$, $\gamma$, $\epsilon$, and $\eta$ are the \emph{sources} of the attack and respective supports, while $\beta$, $\delta$ $\zeta$, and $\theta$ denote their corresponding \emph{targets}.

\subsection{Labelling Semantics and Equational Semantics for HQAF}\label{labequ}
\begin{defn}\label{def:labelling_semantics_hqaf}
	Let $\mathcal{F}=(\mathbf{A}, \mathbf{R_k}, \mathbf{R_n}, \mathbf{R_d}, \mathbf{R_e}, \mathbf{U}, \mathbf{P})$ be an HQAF and let $D\subseteq [0,1]$ be a truth‑value domain.
	A \emph{labelling} of $\mathcal{F}$ over $D$ is a total function
	\[
	\|\cdot\|: \mathbf{U}\to D
	\]
	which assigns to every element $u\in \mathbf{U}$ a truth value $\|u\|\in D$ and always requires for auxiliary elements that $\|\bot\|=0$, $\|\top\|=1$ and $\|\alpha\|=1$ for each $\alpha \in (\mathbf{R}_k \cup \mathbf{R}_n \cup \mathbf{R}_d \cup \mathbf{R}_e) \setminus (\mathbb{R}_k \cup \mathbb{R}_n \cup \mathbb{R}_d \cup \mathbb{R}_e)$. 
\end{defn}

\begin{defn}
	Let $\mathcal{HQ}$ be the set of all HQAF and $\mathcal{LAB}$ be the set of all labellings of all HQAF over $[0,1]$. A labelling semantics for $\mathcal{HQ}$ is a function
	\[
	\mathfrak{LS}: \mathcal{HQ} \to 2^{\mathcal{LAB}}
	\]
	that maps each $\mathcal{F} \in \mathcal{HQ}$ to a set of labellings of the HQAF, denoted by $\mathfrak{LS}(\mathcal{F})$. Each labelling $\|\cdot\| \in \mathfrak{LS}(\mathcal{F})$ is called a model of the HQAF $\mathcal{F}$ under the semantics $\mathfrak{LS}$. The relationship that $\|\cdot\| \in \mathfrak{LS}(\mathcal{F})$ is also denoted by $\|\cdot\| \vDash_\mathfrak{LS}\mathcal{F}$.
\end{defn}
\begin{defn}[Equational System for HQAF]\label{def:equational_system_hqaf}
	Let $\mathcal{F} = (\mathbf{A}, \mathbf{R}_k, \mathbf{R}_n, \mathbf{R}_d, \mathbf{R}_e,$ $\mathbf{U}, \mathbf{P})$ be an HQAF. An equational system $eq$ over $\mathcal{F}$ is defined by the labelling requirements $\|\bot\| = 0$, $\|\top\| = 1$ and $\|\alpha\| = 1$ for each $\alpha \in (\mathbf{R}_k \cup \mathbf{R}_n \cup \mathbf{R}_d \cup \mathbf{R}_e) \setminus (\mathbb{R}_k \cup \mathbb{R}_n \cup \mathbb{R}_d \cup \mathbb{R}_e)$, and for each element $\beta \in \mathbb{U}$ an associated equation of the form
	\[
	\|\beta\| = h_\beta\left( \|x_1\|, \|x_2\|, \dots, \|x_{m_\beta}\| \right), \quad \|\beta\| \in D,
	\]
	where \(\{x_1,x_2,\dots,x_{m_\beta}\}\subseteq \mathbf{U}\) contains exactly the interaction names and source/target elements occurring in the interactions that directly involve \(\beta\): for every attack \(k^\beta_\gamma\in\mathbf{R}_{\mathbf{k}}\) targeting \(\beta\), both the attack name \(k^\beta_\gamma\) and its source \(\gamma\); for every necessary support \(n^\beta_\delta\in\mathbf{R}_{\mathbf{n}}\) targeting \(\beta\), both the support name \(n^\beta_\delta\) and its source \(\delta\); for every deductive support \(d^\theta_\beta\in\mathbf{R}_{\mathbf{d}}\) originating from \(\beta\), both the support name \(d^\theta_\beta\) and its target \(\theta\); and for every evidential support \(e^\beta_\epsilon\in\mathbf{R}_{\mathbf{e}}\) targeting \(\beta\), both the support name \(e^\beta_\epsilon\) and its source \(\epsilon\); and \(h_\beta:D^{m_\beta}\to D\) is a function.
\end{defn}
\begin{defn}[Equational Semantics for HQAF]\label{def:equational_semantics_hqaf}
	An equational semantics for HQAF is a labelling semantics
	\[
	\mathfrak{LS}_{Eq}: \mathcal{HQ} \to 2^{\mathcal{LAB}}
	\]
	such that for every HQAF $\mathcal{F} = (\mathbf{A}, \mathbf{R}_k, \mathbf{R}_n, \mathbf{R}_d, \mathbf{R}_e, \mathbf{U}, \mathbf{P})$ and its associated equational system $eq$,
	\[
	\mathfrak{LS}_{Eq}(\mathcal{F}) = \left\{ \|\cdot\| \mid \|\cdot\| \vDash_{eq} \mathcal{F} \right\}.
	\]
\end{defn}
\begin{defn}[Adjacent Complete Labelling for HQAF]\label{adjcom}
	For an $HQAF=(\mathbf{A},\mathbf{R}_{\mathbf{k}},\mathbf{R}_{\mathbf{n}},\mathbf{R}_{\mathbf{d}},\mathbf{R}_{\mathbf{e}},\mathbf{U},\mathbf{P})$, an adjacent complete labelling is a function $\|\cdot\|:\mathbf{U}\to\{0,\tfrac{1}{2},1\}$ satisfying the auxiliary-element requirements in Definition \ref{def:labelling_semantics_hqaf} and for each $\beta\in\mathbb{U}$:
	\[
	\begin{array}{rl}
		\|\beta\|=1 & \text{ if } 
		\bigl(\forall (\gamma,\beta)\in\mathbf{R}_{\mathbf{k}}:\|\gamma\|=0 \text{ or } \|\mathbf{k}_\gamma^\beta\|=0\bigr) \text{ and } \\
		& \quad \bigl(\forall (\delta,\beta)\in\mathbf{R}_{\mathbf{n}}:\|\delta\|=1 \text{ or } \|\mathbf{n}_\delta^\beta\|=0\bigr) \text{ and } \\
		& \quad \bigl(\forall (\beta,\theta)\in\mathbf{R}_{\mathbf{d}}:\|\theta\|=1 \text{ or } \|\mathbf{d}_\beta^\theta\|=0\bigr) \text{ and } \\
		& \quad \bigl(\exists (\varepsilon,\beta)\in\mathbf{R}_{\mathbf{e}}:\|\varepsilon\|=1 \text{ and } \|\mathbf{e}_\varepsilon^\beta\|=1\bigr); \\[6pt]
		\|\beta\|=0 & \text{ if } 
		\bigl(\exists (\gamma,\beta)\in\mathbf{R}_{\mathbf{k}}:\|\gamma\|=1 \text{ and } \|\mathbf{k}_\gamma^\beta\|=1\bigr) \text{ or } \\
		& \quad \bigl(\exists (\delta,\beta)\in\mathbf{R}_{\mathbf{n}}:\|\delta\|=0 \text{ and } \|\mathbf{n}_\delta^\beta\|=1\bigr) \text{ or } \\
		& \quad \bigl(\exists (\beta,\theta)\in\mathbf{R}_{\mathbf{d}}:\|\theta\|=0 \text{ and } \|\mathbf{d}_\beta^\theta\|=1\bigr) \text{ or } \\
		& \quad \bigl(\forall (\varepsilon,\beta)\in\mathbf{R}_{\mathbf{e}}:\|\varepsilon\|=0 \text{ or } \|\mathbf{e}_\varepsilon^\beta\|=0\bigr); \\[6pt]
		\|\beta\|=\tfrac{1}{2} & \text{ otherwise.}
	\end{array}
	\]
	The adjacent complete labelling semantics, denoted by $\mathfrak{LS}_{ac}$, is defined such that for any HQAF $\mathcal{F}$, $\mathfrak{LS}_{ac}(\mathcal{F})$ is the set of all adjacent complete labellings of $\mathcal{F}$:
	\[
	\mathfrak{LS}_{ac}(\mathcal{F}) = \{\,\|\cdot\| \mid \|\cdot\| \text{ is an adjacent complete labelling of } \mathcal{F}\,\}.
	\]
\end{defn}
We explain the intuition of the adjacent complete labelling. The labelling function $\|\cdot\|\colon \mathbf{U}\to\{0,\frac{1}{2},1\}$ assigns every element (arguments and all higher-order interactions) one of three statuses: $1$ for \textit{valid}, $0$ for \textit{invalid}, and $\frac{1}{2}$ for \textit{undecided}. 

An element $\beta\in\mathbb{U}$ is valid if four conditions are satisfied simultaneously:
for each attack interaction targeting $\beta$, the attack arrow is invalid or the attacker is invalid, which makes the whole attack interaction ineffective;
for each necessary support interaction targeting $\beta$, the support arrow is invalid which makes the interaction ineffective, or the supporter is valid which preserves the validity of $\beta$;
for each deductive support interaction originating from $\beta$, the support arrow is invalid which makes the interaction ineffective, or the target is valid which preserves the validity of $\beta$;
for evidential support interactions targeting $\beta$, at least one valid evidential support interaction exists, i.e., a valid support arrow together with its valid source element.

Correspondingly, $\beta$ is invalid if any destructive condition holds:
there exists an effective attack against $\beta$ with both the attack arrow and the attacker valid;
there exists an unsatisfied necessary support for $\beta$ with a valid arrow but an invalid supporter;
there exists an unsatisfied deductive support from $\beta$ with a valid arrow but an invalid target;
all evidential supports for $\beta$ are ineffective, with no valid arrow-source pair available.

$\beta$ takes the undecided value $\frac{1}{2}$ if neither all validity conditions nor any invalidity condition is fully met.
The term ``adjacent'' reflects that the status of $\beta$ is determined solely by its directly adjacent interactions.

\begin{defn}[Adjacent Stable, Preferred and Grounded Semantics for HQAF]\label{def:adjacent_spg_hqaf}
	Let $\mathcal{F} = (\mathbf{A}, \mathbf{R}_k, \mathbf{R}_n, \mathbf{R}_d, \mathbf{R}_e, \mathbf{U}, \mathbf{P})$ be an HQAF satisfying the labelling requirements in Definition \ref{def:labelling_semantics_hqaf}.
	\begin{enumerate}
		\item A labelling $\|\cdot\|: \mathbf{U} \to \{0, 1\}$ is an adjacent stable labelling of $\mathcal{F}$ if it is an adjacent complete labelling of $\mathcal{F}$ and no element of $\mathbf{U}$ is assigned the value $\frac{1}{2}$. We denote by $\mathfrak{LS}_{as}(\mathcal{F})$ the set of all adjacent stable labellings of $\mathcal{F}$.
		\item An adjacent complete labelling $\|\cdot\|$ of $\mathcal{F}$ is an adjacent preferred labelling if it is maximal with respect to the inclusion order of accepted elements: there does not exist another adjacent complete labelling $\|\cdot\|'$ such that
		\[
		\{x \in \mathbf{U} \mid \|x\| = 1\} \subsetneq \{x \in \mathbf{U} \mid \|x\|' = 1\}.
		\]
		We denote by $\mathfrak{LS}_{ap}(\mathcal{F})$ the set of all adjacent preferred labellings of $\mathcal{F}$.
		\item An adjacent grounded labelling of $\mathcal{F}$ is an adjacent complete labelling $\|\cdot\|$ that is minimal with respect to the inclusion order of accepted elements: there does not exist another adjacent complete labelling $\|\cdot\|'$ such that
		\[
		\{x \in \mathbf{U} \mid \|x\|' = 1\} \subsetneq \{x \in \mathbf{U} \mid \|x\| = 1\}.
		\]
		We denote by $\mathfrak{LS}_{ag}(\mathcal{F})$ the set containing the adjacent grounded labelling.
	\end{enumerate}
\end{defn}
\begin{defn}[Three-Valued Equational Semantics for HQAF]\label{def:3val_equational_hqaf}
	Let $\mathcal{F} = (\mathbf{A}, \mathbf{R}_k, \mathbf{R}_n, \mathbf{R}_d, \mathbf{R}_e, \mathbf{U}, \mathbf{P})$ be an HQAF. For each element $\beta \in \mathbb{U}$, define the following four semantic quantities:
	\begin{align*}
		K(\beta) &\stackrel{\text{def}}{=} \min_{\mathbf{k}_{\gamma}^{\beta} \in \mathbf{R}_k} \max\left\{ 1-\left\| \mathbf{k}_{\gamma}^{\beta} \right\|,\, 1-\left\| \gamma \right\| \right\}, \\
		N(\beta) &\stackrel{\text{def}}{=} \min_{\mathbf{n}_{\delta}^{\beta} \in \mathbf{R}_n} \max\left\{ 1-\left\| \mathbf{n}_{\delta}^{\beta} \right\|,\, \left\| \delta \right\| \right\}, \\
		D(\beta) &\stackrel{\text{def}}{=} \min_{\mathbf{d}_{\beta}^{\theta} \in \mathbf{R}_d} \max\left\{ 1-\left\| \mathbf{d}_{\beta}^{\theta} \right\|,\, \left\| \theta \right\| \right\}, \\
		E(\beta) &\stackrel{\text{def}}{=} \max_{\mathbf{e}_{\varepsilon}^{\beta} \in \mathbf{R}_e} \min\left\{ \left\| \mathbf{e}_{\varepsilon}^{\beta} \right\|,\, \left\| \varepsilon \right\| \right\},
	\end{align*}
	where all $\min$ and $\max$ operations are evaluated over the three-valued domain $\left\{0, \frac{1}{2}, 1\right\}$.
	
	The three-valued equational system $eq_3^{HQ}$ for $\mathcal{F}$ is specified by the auxiliary-element conditions
	\[
	\|\bot\| = 0,\quad \|\top\| = 1,\quad \|\alpha\| = 1 \quad \big(\forall \alpha \in (\mathbf{R}_k \cup \mathbf{R}_n \cup \mathbf{R}_d \cup \mathbf{R}_e) \setminus (\mathbb{R}_k \cup \mathbb{R}_n \cup \mathbb{R}_d \cup \mathbb{R}_e)\big)
	\]
	together with the fixed-point equation for every $\beta \in \mathbb{U}$:
	\[
	\|\beta\| = \min\left\{ K(\beta),\, N(\beta),\, D(\beta),\, E(\beta) \right\}.
	\]
	
	An assignment $\|\cdot\|: \mathbf{U} \to \left\{0, \frac{1}{2}, 1\right\}$ satisfying all equations in $eq_3^{HQ}$ is called a solution of $eq_3^{HQ}$, denoted $\|\cdot\| \vDash_{eq_3^{HQ}} \mathcal{F}$.
	
	The three-valued equational semantics for HQAF induced by $eq_3^{HQ}$ is the labelling semantics
	\[
	\mathfrak{LS}_{Eq_3^{HQ}}: \mathcal{HQ} \to 2^{\mathcal{LAB}},
	\]
	defined by
	\[
	\mathfrak{LS}_{Eq_3^{HQ}}(\mathcal{F}) = \left\{ \|\cdot\| \mid \|\cdot\| \vDash_{eq_3^{HQ}} \mathcal{F} \right\}.
	\]
\end{defn}
The equivalence between three-valued equational semantics and adjacent complete semantics will be proved in the next section.

\section{Encoded Semantics for HQAF}\label{sec:encoding}
\subsection{General Encoded Semantics}

We adopt the following notational convention throughout this section. For any propositional logic $\mathcal{PL}$ with propositional constant set $\mathit{Con}$ and propositional variable set $\mathit{Var}$, and for every HQAF $\mathcal{F} = (\mathbf{A}, \mathbf{R}_k, \mathbf{R}_n, \mathbf{R}_d, \mathbf{R}_e,$ $\mathbf{U}, \mathbf{P})$, we assume the universal set $\mathbf{U}$ is embedded in the signature of $\mathcal{PL}$:
\[
\mathbf{U} \subseteq \mathit{Con} \cup \mathit{Var}, \text{with} \{\bot, \top\} \cup \big( (\mathbf{R}_k \cup \mathbf{R}_n \cup \mathbf{R}_d \cup \mathbf{R}_e) \setminus (\mathbb{R}_k \cup \mathbb{R}_n \cup \mathbb{R}_d \cup \mathbb{R}_e) \big) \subseteq \mathit{Con}.
\]
Auxiliary elements receive fixed interpretations: $\|\bot\| = 0$, $\|\top\| = 1$, and $\|\alpha\| = 1$ for every auxiliary interaction $\alpha \in (\mathbf{R}_k \cup \mathbf{R}_n \cup \mathbf{R}_d \cup \mathbf{R}_e) \setminus (\mathbb{R}_k \cup \mathbb{R}_n \cup \mathbb{R}_d \cup \mathbb{R}_e)$. This convention identifies HQAF auxiliary elements with logical constants, providing a uniform semantic foundation for all encoding schemes below.

\begin{defn}[Encoding for HQAF]\label{def:encoding_hqaf}
	An encoding of HQAF w.r.t. a propositional logic $\mathcal{PL}$ is a function
	\[
	ec: \mathcal{HQ} \to F_{\mathcal{PL}}
	\]
	such that for every HQAF $\mathcal{F} = (\mathbf{A}, \mathbf{R}_k, \mathbf{R}_n, \mathbf{R}_d, \mathbf{R}_e, \mathbf{U}, \mathbf{P})$, the set of propositional variables and constants appearing in $ec(\mathcal{F})$ coincides exactly with the universal set $\mathbf{U}$. The formula $ec(\mathcal{F})$ is called the \emph{encoded formula} of the HQAF $\mathcal{F}$.
\end{defn}

\begin{defn}[Encoded Semantics]\label{def:encoded_semantics_hqaf}
	Given a propositional logic system $\mathcal{PL}$ and an encoding function $ec: \mathcal{HQ} \to F_{\mathcal{PL}}$, the \emph{encoded semantics} induced by $ec$ and $\mathcal{PL}$ is a labelling semantics $\mathfrak{LS}_{ec}^{\mathcal{PL}}$ defined by
	\[
	\mathfrak{LS}_{ec}^{\mathcal{PL}}(\mathcal{F}) = \left\{ \|\cdot\| \mid \left\| ec(\mathcal{F}) \right\| = 1 \right\}.
	\]
	An assignment $\|\cdot\| \in \mathfrak{LS}_{ec}^{\mathcal{PL}}(\mathcal{F})$ is called a \emph{model} of the HQAF under this encoded semantics, denoted $\|\cdot\| \vDash_{\mathfrak{LS}_{ec}^{\mathcal{PL}}} \mathcal{F}$, or equivalently $\|\cdot\| \vDash_{\mathcal{PL}} ec(\mathcal{F})$.
\end{defn}

\begin{defn}[Skeptical and Credulous Encoded Semantics for HQAF]\label{def:skeptical_credulous_hqaf}
	Let $\mathfrak{LS}_{ec}^{\mathcal{PL}}$ be an encoded semantics associated with a propositional logic $\mathcal{PL}$ with truth value set $L \subseteq [0,1]$ and an encoding function $ec$.
	
	The \emph{skeptical encoded semantics} induced by $ec$ and $\mathcal{PL}$ is the labelling semantics $\mathfrak{LS}_{ec}^{\mathcal{PL}, sk}$ such that for any HQAF $\mathcal{F}$,
	\[
	\mathfrak{LS}_{ec}^{\mathcal{PL}, sk}(\mathcal{F}) = \left\{ \|\cdot\|_{\mathcal{F}}^{sk}: \mathbf{U} \to L \mid \|x\|_{\mathcal{F}}^{sk} = \min\left\{ \|x\| \mid \|\cdot\| \in \mathfrak{LS}_{ec}^{\mathcal{PL}}(\mathcal{F}) \right\},\ \forall x \in \mathbf{U} \right\}.
	\]
	
	Analogously, the \emph{credulous encoded semantics} induced by $ec$ and $\mathcal{PL}$ is the labelling semantics $\mathfrak{LS}_{ec}^{\mathcal{PL}, cr}$ such that for any HQAF $\mathcal{F}$,
	\[
	\mathfrak{LS}_{ec}^{\mathcal{PL}, cr}(\mathcal{F}) = \left\{ \|\cdot\|_{\mathcal{F}}^{cr}: \mathbf{U} \to L \mid \|x\|_{\mathcal{F}}^{cr} = \max\left\{ \|x\| \mid \|\cdot\| \in \mathfrak{LS}_{ec}^{\mathcal{PL}}(\mathcal{F}) \right\},\ \forall x \in \mathbf{U} \right\}.
	\]
\end{defn}
Combining the normal encoded formulas of NHAF and EHAF and adopting the view that an element A deductively supports an element B iff B necessarily supports A, we obtain the normal encoded formula for HQAF.
\begin{defn}[Normal Encoding of HQAF]\label{def:normal_encoding_hqaf}
	For a given propositional logic $\mathcal{PL}$, the \emph{normal encoding} of HQAF w.r.t. $\mathcal{PL}$ is the encoding function $ec_n$ such that for every HQAF $\mathcal{F} = (\mathbf{A}, \mathbf{R}_k, \mathbf{R}_n, \mathbf{R}_d, \mathbf{R}_e, \mathbf{U}, \mathbf{P})$,
	\begin{align*}
		ec_n(\mathcal{F}) &= \bigwedge_{\beta \in \mathbb{U}} \Bigl( \beta \leftrightarrow
		\left( \bigwedge_{\mathbf{k}_\gamma^\beta \in \mathbf{R}_k} \neg \left( \mathbf{k}_\gamma^\beta \land \gamma \right) \right) \land \left( \bigwedge_{\mathbf{n}_\delta^\beta \in \mathbf{R}_n} \neg \left( \mathbf{n}_\delta^\beta \land \neg \delta \right) \right) \land {} \\
		&\qquad \left( \bigwedge_{\mathbf{d}_\beta^\theta \in \mathbf{R}_d} \neg \left( \mathbf{d}_\beta^\theta \land \neg \theta \right) \right) \land \left( \bigvee_{\mathbf{e}_\varepsilon^\beta \in \mathbf{R}_e} \left( \mathbf{e}_\varepsilon^\beta \land \varepsilon \right) \right) \Bigr).
	\end{align*}
	The \emph{normal encoded semantics} w.r.t. $\mathcal{PL}$ is the encoded semantics induced by $ec_n$ and $\mathcal{PL}$, denoted by $\mathfrak{LS}_{ec_n}^{\mathcal{PL}}$.
\end{defn}

\subsection{Discrete Normal Encoded Semantics}

\begin{thm}\label{thm:3val_encoded_equivalence_hqaf}
	For any HQAF $\mathcal{F} = (\mathbf{A}, \mathbf{R}_k, \mathbf{R}_n, \mathbf{R}_d, \mathbf{R}_e, \mathbf{U}, \mathbf{P})$, 
	\[
	\mathfrak{LS}_{ec_n}^{\mathcal{PL}_3^L}(\mathcal{F}) = \mathfrak{LS}_{ac}(\mathcal{F}).
	\]
\end{thm}

\begin{proof}
	Any assignment on auxiliary elements in $\mathbf{U}$ is trivially consistent between adjacent complete labellings and models of $ec_n(\mathcal{F})$ in $\mathcal{PL}_3^L$. We need to check that for a given assignment $\|\cdot\|$ and for each $\beta \in \mathbb{U}$, the value of $\beta$ satisfies the adjacent complete labelling conditions (Definition~\ref{adjcom}) iff
	\[
	\left\| \beta \leftrightarrow \left(
	\begin{aligned}
		&\bigwedge_{\mathbf{k}_{\gamma}^{\beta} \in \mathbf{R}_k} \neg \left( \mathbf{k}_{\gamma}^{\beta} \land \gamma \right) \land
		\bigwedge_{\mathbf{n}_{\delta}^{\beta} \in \mathbf{R}_n} \neg \left( \mathbf{n}_{\delta}^{\beta} \land \neg \delta \right) \\
		&\land \bigwedge_{\mathbf{d}_{\beta}^{\theta} \in \mathbf{R}_d} \neg \left( \mathbf{d}_{\beta}^{\theta} \land \neg \theta \right) \land
		\bigvee_{\mathbf{e}_{\varepsilon}^{\beta} \in \mathbf{R}_e} \left( \mathbf{e}_{\varepsilon}^{\beta} \land \varepsilon \right)
	\end{aligned}
	\right) \right\| = 1
	\]
	in $\mathcal{PL}_3^L$. Denote the right-hand side of the equivalence inside the formula by $\Phi_\beta$. We discuss three cases.
	
	\begin{itemize}
		\item \textbf{Case 1, $\|\beta\| = 1$.}
		\\
		$\|\beta\| = 1$ by adjacent complete labelling
		\\
		$\Longleftrightarrow$ $\left[\forall \mathbf{k}_{\gamma}^{\beta} \in \mathbf{R}_k: \|\mathbf{k}_{\gamma}^{\beta}\| = 0 \text{ or } \|\gamma\| = 0\right]$ and $\left[\forall \mathbf{n}_{\delta}^{\beta} \in \mathbf{R}_n: \|\mathbf{n}_{\delta}^{\beta}\| = 0 \text{ or } \|\delta\| = 1\right]$ and $\left[\forall \mathbf{d}_{\beta}^{\theta} \in \mathbf{R}_d: \|\mathbf{d}_{\beta}^{\theta}\| = 0 \text{ or } \|\theta\| = 1\right]$ and $\left[\exists \mathbf{e}_{\varepsilon}^{\beta} \in \mathbf{R}_e: \|\mathbf{e}_{\varepsilon}^{\beta}\| = 1 \text{ and } \|\varepsilon\| = 1\right]$ by Definition~\ref{adjcom}
		\\
		$\Longleftrightarrow$ $\left[\forall \mathbf{k}_{\gamma}^{\beta} \in \mathbf{R}_k: \left\| \mathbf{k}_{\gamma}^{\beta} \land \gamma \right\| = 0\right]$ and $\left[\forall \mathbf{n}_{\delta}^{\beta} \in \mathbf{R}_n: \left\| \mathbf{n}_{\delta}^{\beta} \land \neg \delta \right\| = 0\right]$ and $\left[\forall \mathbf{d}_{\beta}^{\theta} \in \mathbf{R}_d: \left\| \mathbf{d}_{\beta}^{\theta} \land \neg \theta \right\| = 0\right]$ and $\left[\exists \mathbf{e}_{\varepsilon}^{\beta} \in \mathbf{R}_e: \left\| \mathbf{e}_{\varepsilon}^{\beta} \land \varepsilon \right\| = 1\right]$ in $\mathcal{PL}_3^L$
		\\
		$\Longleftrightarrow$ $\left[\forall \mathbf{k}_{\gamma}^{\beta} \in \mathbf{R}_k: \left\| \neg \left( \mathbf{k}_{\gamma}^{\beta} \land \gamma \right) \right\| = 1\right]$ and $\left[\forall \mathbf{n}_{\delta}^{\beta} \in \mathbf{R}_n: \left\| \neg \left( \mathbf{n}_{\delta}^{\beta} \land \neg \delta \right) \right\| = 1\right]$ and $\left[\forall \mathbf{d}_{\beta}^{\theta} \in \mathbf{R}_d: \left\| \neg \left( \mathbf{d}_{\beta}^{\theta} \land \neg \theta \right) \right\| = 1\right]$ and $\left\| \bigvee_{\mathbf{e}_{\varepsilon}^{\beta} \in \mathbf{R}_e} \left( \mathbf{e}_{\varepsilon}^{\beta} \land \varepsilon \right) \right\| = 1$ in $\mathcal{PL}_3^L$
		\\
		$\Longleftrightarrow$ $\left\| \bigwedge_{\mathbf{k}_{\gamma}^{\beta} \in \mathbf{R}_k} \neg \left( \mathbf{k}_{\gamma}^{\beta} \land \gamma \right) \right\| = 1$ and $\left\| \bigwedge_{\mathbf{n}_{\delta}^{\beta} \in \mathbf{R}_n} \neg \left( \mathbf{n}_{\delta}^{\beta} \land \neg \delta \right) \right\| = 1$ and $\left\| \bigwedge_{\mathbf{d}_{\beta}^{\theta} \in \mathbf{R}_d} \neg \left( \mathbf{d}_{\beta}^{\theta} \land \neg \theta \right) \right\| = 1$ and $\left\| \bigvee_{\mathbf{e}_{\varepsilon}^{\beta} \in \mathbf{R}_e} \left( \mathbf{e}_{\varepsilon}^{\beta} \land \varepsilon \right) \right\| = 1$ in $\mathcal{PL}_3^L$
		\\
		$\Longleftrightarrow$ $\|\Phi_\beta\| = 1$ in $\mathcal{PL}_3^L$
		\\
		$\Longleftrightarrow$ $\|\beta \leftrightarrow \Phi_\beta\| = 1$ in $\mathcal{PL}_3^L$.
		
		\item \textbf{Case 2, $\|\beta\| = 0$.}
		\\
		$\|\beta\| = 0$ by adjacent complete labelling
		\\
		$\Longleftrightarrow$ $\left[\exists \mathbf{k}_{\gamma}^{\beta} \in \mathbf{R}_k: \|\mathbf{k}_{\gamma}^{\beta}\| = 1 \text{ and } \|\gamma\| = 1\right]$ or $\left[\exists \mathbf{n}_{\delta}^{\beta} \in \mathbf{R}_n: \|\mathbf{n}_{\delta}^{\beta}\| = 1 \text{ and } \|\delta\| = 0\right]$ or $\left[\exists \mathbf{d}_{\beta}^{\theta} \in \mathbf{R}_d: \|\mathbf{d}_{\beta}^{\theta}\| = 1 \text{ and } \|\theta\| = 0\right]$ or $\left[\forall \mathbf{e}_{\varepsilon}^{\beta} \in \mathbf{R}_e: \|\mathbf{e}_{\varepsilon}^{\beta}\| = 0 \text{ or } \|\varepsilon\| = 0\right]$ by Definition~\ref{adjcom}
		\\
		$\Longleftrightarrow$ $\left[\exists \mathbf{k}_{\gamma}^{\beta} \in \mathbf{R}_k: \left\| \mathbf{k}_{\gamma}^{\beta} \land \gamma \right\| = 1\right]$ or $\left[\exists \mathbf{n}_{\delta}^{\beta} \in \mathbf{R}_n: \left\| \mathbf{n}_{\delta}^{\beta} \land \neg \delta \right\| = 1\right]$ \\or $\left[\exists \mathbf{d}_{\beta}^{\theta} \in \mathbf{R}_d: \left\| \mathbf{d}_{\beta}^{\theta} \land \neg \theta \right\| = 1\right]$ or $\left[\forall \mathbf{e}_{\varepsilon}^{\beta} \in \mathbf{R}_e: \left\| \mathbf{e}_{\varepsilon}^{\beta} \land \varepsilon \right\| = 0\right]$ in $\mathcal{PL}_3^L$
		\\
		$\Longleftrightarrow$ $\left[\exists \mathbf{k}_{\gamma}^{\beta} \in \mathbf{R}_k: \left\| \neg \left( \mathbf{k}_{\gamma}^{\beta} \land \gamma \right) \right\| = 0\right]$ or $\left[\exists \mathbf{n}_{\delta}^{\beta} \in \mathbf{R}_n: \left\| \neg \left( \mathbf{n}_{\delta}^{\beta} \land \neg \delta \right) \right\| = 0\right]$ or $\left[\exists \mathbf{d}_{\beta}^{\theta} \in \mathbf{R}_d: \left\| \neg \left( \mathbf{d}_{\beta}^{\theta} \land \neg \theta \right) \right\| = 0\right]$ or $\left\| \bigvee_{\mathbf{e}_{\varepsilon}^{\beta} \in \mathbf{R}_e} \left( \mathbf{e}_{\varepsilon}^{\beta} \land \varepsilon \right) \right\| = 0$ in $\mathcal{PL}_3^L$
		\\
		$\Longleftrightarrow$ $\left\| \bigwedge_{\mathbf{k}_{\gamma}^{\beta} \in \mathbf{R}_k} \neg \left( \mathbf{k}_{\gamma}^{\beta} \land \gamma \right) \right\| = 0$ or $\left\| \bigwedge_{\mathbf{n}_{\delta}^{\beta} \in \mathbf{R}_n} \neg \left( \mathbf{n}_{\delta}^{\beta} \land \neg \delta \right) \right\| = 0$ \\or $\left\| \bigwedge_{\mathbf{d}_{\beta}^{\theta} \in \mathbf{R}_d} \neg \left( \mathbf{d}_{\beta}^{\theta} \land \neg \theta \right) \right\| = 0$ or $\left\| \bigvee_{\mathbf{e}_{\varepsilon}^{\beta} \in \mathbf{R}_e} \left( \mathbf{e}_{\varepsilon}^{\beta} \land \varepsilon \right) \right\| = 0$ in $\mathcal{PL}_3^L$
		\\
		$\Longleftrightarrow$ $\|\Phi_\beta\| = 0$ in $\mathcal{PL}_3^L$
		\\
		$\Longleftrightarrow$ $\|\beta \leftrightarrow \Phi_\beta\| = 1$ in $\mathcal{PL}_3^L$.
		
		\item \textbf{Case 3, $\|\beta\| = \frac{1}{2}$.}
		\\
		$\|\beta\| = \frac{1}{2}$ by adjacent complete labelling
		\\
		$\Longleftrightarrow$ $\|\beta\| \neq 1$ and $\|\beta\| \neq 0$ under adjacent complete labelling
		\\
		$\Longleftrightarrow$ $\|\Phi_\beta\| \neq 1$ and $\|\Phi_\beta\| \neq 0$ in $\mathcal{PL}_3^L$ by Case 1 and Case 2
		\\
		$\Longleftrightarrow$ $\|\Phi_\beta\| = \frac{1}{2}$ in $\mathcal{PL}_3^L$
		\\
		$\Longleftrightarrow$ $\|\beta \leftrightarrow \Phi_\beta\| = 1$ in $\mathcal{PL}_3^L$.
	\end{itemize}
	
	From the three cases above, for a given assignment $\|\cdot\|$, the value $\|\beta\|$ of any $\beta \in \mathbb{U}$ satisfies the adjacent complete labelling conditions iff
	\[
	\left\| \beta \leftrightarrow \left(
	\begin{aligned}
		&\bigwedge_{\mathbf{k}_{\gamma}^{\beta} \in \mathbf{R}_k} \neg \left( \mathbf{k}_{\gamma}^{\beta} \land \gamma \right) \land
		\bigwedge_{\mathbf{n}_{\delta}^{\beta} \in \mathbf{R}_n} \neg \left( \mathbf{n}_{\delta}^{\beta} \land \neg \delta \right) \\
		&\land \bigwedge_{\mathbf{d}_{\beta}^{\theta} \in \mathbf{R}_d} \neg \left( \mathbf{d}_{\beta}^{\theta} \land \neg \theta \right) \land
		\bigvee_{\mathbf{e}_{\varepsilon}^{\beta} \in \mathbf{R}_e} \left( \mathbf{e}_{\varepsilon}^{\beta} \land \varepsilon \right)
	\end{aligned}
	\right) \right\| = 1
	\]
	in $\mathcal{PL}_3^L$.
	
	Thus, an assignment $\|\cdot\|$ of the HQAF is an adjacent complete labelling iff the assignment $\|\cdot\|$ is a model of $ec_n(\mathcal{F})$ in $\mathcal{PL}_3^L$.
\end{proof}
\begin{thm}\label{thm:3eq_encoded_equivalence_hqaf}
	For every HQAF $\mathcal{F}$,
	\[
	\mathfrak{LS}_{Eq_3^{HQ}}(\mathcal{F}) = \mathfrak{LS}_{ec_n}^{\mathcal{PL}_3^L}(\mathcal{F}).
	\]
\end{thm}

\begin{proof}
	Fix an HQAF $\mathcal{F}=(\mathbf{A},\mathbf{R}_k,\mathbf{R}_n,\mathbf{R}_d,\mathbf{R}_e,\mathbf{U},\mathbf{P})$ and an assignment $\|\cdot\|:\mathbf{U}\to\{0,\frac{1}{2},1\}$. We prove that $\|\cdot\|$ satisfies the three-valued equational system $eq_3^{HQ}$ if and only if $\|\cdot\|$ is a model of the normal encoding $ec_n(\mathcal{F})$ in $\mathcal{PL}_3^L$, i.e., $\|ec_n(\mathcal{F})\|_{\mathcal{PL}_3^L}=1$. Since the auxiliary-element conditions are identical in both systems, it suffices to establish the equivalence for every $\beta\in\mathbb{U}$ between the fixed-point equation of $eq_3^{HQ}$ and the corresponding equivalence in $ec_n(\mathcal{F})$.
	
	Recall that in $\mathcal{PL}_3^L$ the connectives are interpreted as
	\[
	\|\neg a\|=1-\|a\|,\qquad
	\|a\wedge b\|=\min\{\|a\|,\|b\|\},\qquad
	\|a\vee b\|=\max\{\|a\|,\|b\|\},
	\]
	and the equivalence connective satisfies
	\[
	\|a\leftrightarrow b\|=1 \quad\Longleftrightarrow\quad \|a\|=\|b\|.
	\]
	The normal encoding is
	\[
	ec_n(\mathcal{F})=\bigwedge_{\beta\in\mathbb{U}}\bigl(\beta\leftrightarrow\Phi_\beta\bigr),
	\]
	where
	\begin{align*}
		\Phi_\beta=&
		\left(\bigwedge_{\mathbf{k}_{\gamma}^{\beta}\in\mathbf{R}_k}\neg(\mathbf{k}_{\gamma}^{\beta}\wedge\gamma)\right)
		\wedge
		\left(\bigwedge_{\mathbf{n}_{\delta}^{\beta}\in\mathbf{R}_n}\neg(\mathbf{n}_{\delta}^{\beta}\wedge\neg\delta)\right)
		\wedge\\
		&\left(\bigwedge_{\mathbf{d}_{\beta}^{\theta}\in\mathbf{R}_d}\neg(\mathbf{d}_{\beta}^{\theta}\wedge\neg\theta)\right)
		\wedge
		\left(\bigvee_{\mathbf{e}_{\varepsilon}^{\beta}\in\mathbf{R}_e}(\mathbf{e}_{\varepsilon}^{\beta}\wedge\varepsilon)\right).
	\end{align*}
	Thus $\|ec_n(\mathcal{F})\|=1$ if and only if for every $\beta\in\mathbb{U}$, $\|\beta\leftrightarrow\Phi_\beta\|=1$, which is equivalent to $\|\beta\|=\|\Phi_\beta\|$.
	
	Now fix $\beta\in\mathbb{U}$. We compute the truth value of $\Phi_\beta$ in $\mathcal{PL}_3^L$ part by part. For the attack part, using $\|\neg a\|=1-\|a\|$ and $\|a\wedge b\|=\min\{\|a\|,\|b\|\}$, we obtain
	\[
	\begin{aligned}
		\left\|\bigwedge_{\mathbf{k}_{\gamma}^{\beta}\in\mathbf{R}_k}\neg(\mathbf{k}_{\gamma}^{\beta}\wedge\gamma)\right\|
		&= \min_{\mathbf{k}_{\gamma}^{\beta}\in\mathbf{R}_k} \bigl(1-\min\{\|\mathbf{k}_{\gamma}^{\beta}\|,\|\gamma\|\}\bigr) \\
		&= \min_{\mathbf{k}_{\gamma}^{\beta}\in\mathbf{R}_k} \max\{1-\|\mathbf{k}_{\gamma}^{\beta}\|,1-\|\gamma\|\} \\
		&= K(\beta).
	\end{aligned}
	\]
	For the necessary support part, using $\|\neg\delta\|=1-\|\delta\|$, we have
	\[
	\begin{aligned}
		\left\|\bigwedge_{\mathbf{n}_{\delta}^{\beta}\in\mathbf{R}_n}\neg(\mathbf{n}_{\delta}^{\beta}\wedge\neg\delta)\right\|
		&= \min_{\mathbf{n}_{\delta}^{\beta}\in\mathbf{R}_n} \bigl(1-\min\{\|\mathbf{n}_{\delta}^{\beta}\|,1-\|\delta\|\}\bigr) \\
		&= \min_{\mathbf{n}_{\delta}^{\beta}\in\mathbf{R}_n} \max\{1-\|\mathbf{n}_{\delta}^{\beta}\|,\|\delta\|\} \\
		&= N(\beta).
	\end{aligned}
	\]
	Similarly, the deductive support part evaluates to
	\[
	\begin{aligned}
		\left\|\bigwedge_{\mathbf{d}_{\beta}^{\theta}\in\mathbf{R}_d}\neg(\mathbf{d}_{\beta}^{\theta}\wedge\neg\theta)\right\|
		&= \min_{\mathbf{d}_{\beta}^{\theta}\in\mathbf{R}_d} \bigl(1-\min\{\|\mathbf{d}_{\beta}^{\theta}\|,1-\|\theta\|\}\bigr) \\
		&= \min_{\mathbf{d}_{\beta}^{\theta}\in\mathbf{R}_d} \max\{1-\|\mathbf{d}_{\beta}^{\theta}\|,\|\theta\|\} \\
		&= D(\beta).
	\end{aligned}
	\]
	Finally, the evidential support part, using $\|a\vee b\|=\max\{\|a\|,\|b\|\}$ and $\|a\wedge b\|=\min\{\|a\|,\|b\|\}$, gives
	\[
	\left\|\bigvee_{\mathbf{e}_{\varepsilon}^{\beta}\in\mathbf{R}_e}(\mathbf{e}_{\varepsilon}^{\beta}\wedge\varepsilon)\right\|
	= \max_{\mathbf{e}_{\varepsilon}^{\beta}\in\mathbf{R}_e} \min\{\|\mathbf{e}_{\varepsilon}^{\beta}\|,\|\varepsilon\|\}
	= E(\beta).
	\]
	Since $\Phi_\beta$ is the conjunction of these four parts and conjunction is interpreted as the minimum, we conclude
	\[
	\|\Phi_\beta\|
	= \min\{K(\beta),N(\beta),D(\beta),E(\beta)\}.
	\]
	Therefore, for every $\beta\in\mathbb{U}$,
	\[
	\|\beta\|=\min\{K(\beta),N(\beta),D(\beta),E(\beta)\}
	\quad\Longleftrightarrow\quad
	\|\beta\|=\|\Phi_\beta\|.
	\]
	By the property of the equivalence connective in $\mathcal{PL}_3^L$, the latter is equivalent to
	\[
	\|\beta\leftrightarrow\Phi_\beta\|=1.
	\]
	Hence the fixed-point equation of $eq_3^{HQ}$ for $\beta$ holds if and only if the corresponding conjunct $\beta\leftrightarrow\Phi_\beta$ in $ec_n(\mathcal{F})$ has truth value $1$.
	
	Since $ec_n(\mathcal{F})$ is the conjunction of all these conjuncts, its truth value is $1$ if and only if every conjunct has truth value $1$. Consequently, the assignment $\|\cdot\|$ satisfies all equations of $eq_3^{HQ}$ if and only if $\|ec_n(\mathcal{F})\|_{\mathcal{PL}_3^L}=1$. Therefore,
	\[
	\|\cdot\|\vDash_{eq_3^{HQ}}\mathcal{F}
	\quad\Longleftrightarrow\quad
	\|\cdot\|\vDash_{\mathfrak{LS}_{ec_n}^{\mathcal{PL}_3^L}}\mathcal{F},
	\]
	and the two labelling semantics coincide:
	\[
	\mathfrak{LS}_{Eq_3^{HQ}}(\mathcal{F}) = \mathfrak{LS}_{ec_n}^{\mathcal{PL}_3^L}(\mathcal{F}).
	\]
\end{proof}
\begin{cor}\label{thm:3val_equivalence_hqaf}
	For every HQAF $\mathcal{F}$,
	\[
	\mathfrak{LS}_{Eq_3^{HQ}}(\mathcal{F})  = \mathfrak{LS}_{ec_n}^{\mathcal{PL}_3^L}(\mathcal{F})= \mathfrak{LS}_{ac}(\mathcal{F}).
	\]
\end{cor}
\begin{proof}
	It follows immediately from Theorem~\ref{thm:3val_encoded_equivalence_hqaf} and Theorem \ref{thm:3eq_encoded_equivalence_hqaf}.
\end{proof}

\begin{thm}
	For every HQAF $\mathcal{F}$,
	\[
	\mathfrak{LS}_{as}(\mathcal{F}) = \mathfrak{LS}_{ec_n}^{\mathcal{PL}_2}(\mathcal{F}).
	\]
\end{thm}
\begin{proof}
	Let $\mathcal{F}$ be an arbitrary HQAF. By Theorem~\ref{thm:3val_encoded_equivalence_hqaf}, the normal encoding $ec_n$ exactly captures the adjacent complete semantics in the three-valued logic $\mathcal{PL}_3^L$:
	\[
	\mathfrak{LS}_{ec_n}^{\mathcal{PL}_3^L}(\mathcal{F}) = \mathfrak{LS}_{ac}(\mathcal{F}).
	\]
	The encoding formula $ec_n(\mathcal{F})$ uses only the connectives $\neg, \land, \vee, \leftrightarrow$, whose truth tables on the set $\{0,1\}$ coincide with those of classical two-valued logic $\mathcal{PL}_2$. Therefore, restricting the universe of assignments to $\{0,1\}$, we have
	\[
	\mathfrak{LS}_{ec_n}^{\mathcal{PL}_2}(\mathcal{F}) 
	= \bigl\{ \|\cdot\| \in \mathfrak{LS}_{ec_n}^{\mathcal{PL}_3^L}(\mathcal{F}) \mid \|\cdot\| : \mathbf{U} \to \{0,1\} \bigr\}.
	\]
	By Definition~\ref{def:adjacent_spg_hqaf}, a labelling is adjacent stable exactly when it is adjacent complete and assigns only the values $0$ or $1$ to every element of $\mathbf{U}$. Hence
	\[
	\mathfrak{LS}_{as}(\mathcal{F})
	= \bigl\{ \|\cdot\| \in \mathfrak{LS}_{ac}(\mathcal{F}) \mid \|\cdot\| : \mathbf{U} \to \{0,1\} \bigr\}.
	\]
	Combining the two equalities yields
	\[
	\mathfrak{LS}_{ec_n}^{\mathcal{PL}_2}(\mathcal{F}) = \mathfrak{LS}_{as}(\mathcal{F}).
	\]
	Since $\mathcal{F}$ was arbitrary, the theorem follows. \qedhere
\end{proof}

\subsection{General Fuzzy Encoded Semantics}

\subsubsection{Continuous Fuzzy Operator-Based Equational (CFOE) Semantics}

\begin{defn}[CFOE System for HQAF]\label{def:cfoe_system_hqaf}
	Let $N^*$ be a continuous negation and $\otimes: [0,1]^2 \to [0,1]$ be a continuous t-norm. For an HQAF $\mathcal{F} = (\mathbf{A}, \mathbf{R}_k, \mathbf{R}_n, \mathbf{R}_d, \mathbf{R}_e, \mathbf{U}, \mathbf{P})$, denote by $\mathcal{K}_\beta = \{\mathbf{k}_{\gamma}^{\beta} \in \mathbf{R}_k\}$ the set of attacks targeting $\beta$, $\mathcal{N}_\beta = \{\mathbf{n}_{\delta}^{\beta} \in \mathbf{R}_n\}$ the set of necessary supports targeting $\beta$, $\mathcal{D}_\beta = \{\mathbf{d}_{\beta}^{\theta} \in \mathbf{R}_d\}$ the set of deductive supports originating from $\beta$, and $\mathcal{E}_\beta = \{\mathbf{e}_{\varepsilon}^{\beta} \in \mathbf{R}_e\}$ the set of evidential supports targeting $\beta$. Let $\|\cdot\|: \mathbf{U} \to [0,1]$ be an assignment.
	
	A continuous fuzzy operator-based equational (CFOE) system $eq_{[0,1]}^{*, \otimes}$ for $\mathcal{F}$ is defined by the auxiliary-element conditions
	\[
	\|\bot\| = 0,\quad \|\top\| = 1,\quad \|\alpha\| = 1 \quad \big(\forall \alpha \in (\mathbf{R}_k \cup \mathbf{R}_n \cup \mathbf{R}_d \cup \mathbf{R}_e) \setminus (\mathbb{R}_k \cup \mathbb{R}_n \cup \mathbb{R}_d \cup \mathbb{R}_e)\big)
	\]
	and for every $\beta \in \mathbb{U}$, the fixed-point equation:
	\[
	\|\beta\| = 
	\begin{aligned}[t]
		&\left( \bigotimes_{\mathbf{k}_{\gamma}^{\beta} \in \mathcal{K}_\beta} N^*\left( \left\| \mathbf{k}_{\gamma}^{\beta} \right\| \otimes \left\| \gamma \right\| \right) \right) \\
		&\otimes \left( \bigotimes_{\mathbf{n}_{\delta}^{\beta} \in \mathcal{N}_\beta} N^*\left( \left\| \mathbf{n}_{\delta}^{\beta} \right\| \otimes N^*\left( \left\| \delta \right\| \right) \right) \right) \\
		&\otimes \left( \bigotimes_{\mathbf{d}_{\beta}^{\theta} \in \mathcal{D}_\beta} N^*\left( \left\| \mathbf{d}_{\beta}^{\theta} \right\| \otimes N^*\left( \left\| \theta \right\| \right) \right) \right) \\
		&\otimes \left( N^*\left( \bigotimes_{\mathbf{e}_{\varepsilon}^{\beta} \in \mathcal{E}_\beta} N^*\left( \left\| \mathbf{e}_{\varepsilon}^{\beta} \right\| \otimes \left\| \varepsilon \right\| \right) \right) \right).
	\end{aligned}
	\]
	Here $\bigotimes$ denotes the finite iteration of the t-norm $\otimes$, and the last factor employs the standard t-conorm defined by $x \oplus y = N^*(N^*(x) \otimes N^*(y))$, so that the disjunction over evidential supports is expressed as $N^*\left( \bigotimes_{\mathbf{e} \in \mathcal{E}_\beta} N^*\left( \left\| \mathbf{e}_{\varepsilon}^{\beta} \right\| \otimes \left\| \varepsilon \right\| \right) \right)$.
	
	An assignment $\|\cdot\|: \mathbf{U} \to [0,1]$ satisfying all equations is called a solution of $eq_{[0,1]}^{*, \otimes}$, denoted $\|\cdot\| \vDash_{eq_{[0,1]}^{*, \otimes}} \mathcal{F}$.
	
	The continuous fuzzy operator-based equational (CFOE) semantics induced by $eq_{[0,1]}^{*, \otimes}$ is the function
	\[
	\mathfrak{LS}_{Eq_{[0,1]}^{*, \otimes}}: \mathcal{HQ} \to 2^{\mathcal{LAB}}, \quad
	\mathfrak{LS}_{Eq_{[0,1]}^{*, \otimes}}(\mathcal{F}) = \left\{ \|\cdot\| \mid \|\cdot\| \vDash_{eq_{[0,1]}^{*, \otimes}} \mathcal{F} \right\}.
	\]
\end{defn}

\subsubsection{Continuous Fuzzy Normal Encoded (CFNE) Semantics}

\begin{defn}[CFNE Semantics for HQAF]\label{def:cfne_semantics_hqaf}
	Let $\mathcal{PL}_{[0,1]}^{*, \otimes}$ be a fuzzy propositional logic system equipped with a continuous negation $N^*$, a continuous t-norm $\otimes$, and its residuated implication. The continuous fuzzy normal encoded (CFNE) semantics is the encoded semantics induced by the normal encoding $ec_n$ (Definition~\ref{def:normal_encoding_hqaf}) and $\mathcal{PL}_{[0,1]}^{*, \otimes}$, denoted by $\mathfrak{LS}_{ec_n}^{\mathcal{PL}_{[0,1]}^*}$.
	
	That is, for every HQAF $\mathcal{F}$,
	\[
	\mathfrak{LS}_{ec_n}^{\mathcal{PL}_{[0,1]}^*}(\mathcal{F}) = \left\{ \|\cdot\| : \mathbf{U} \to [0,1] \mid \left\| ec_n(\mathcal{F}) \right\|_{\mathcal{PL}_{[0,1]}^{*, \otimes}} = 1 \right\}.
	\]
\end{defn}

\subsubsection{Equivalence between CFOE and CFNE Semantics}

\begin{thm}[Equivalence of CFOE and CFNE Semantics]\label{thm:cfoe_cfne_equivalence_hqaf}
	For every HQAF $\mathcal{F} \in \mathcal{HQ}$,
	\[
	\mathfrak{LS}_{Eq_{[0,1]}^{*, \otimes}}(\mathcal{F}) = \mathfrak{LS}_{ec_n}^{\mathcal{PL}_{[0,1]}^*}(\mathcal{F}).
	\]
\end{thm}

\begin{proof}
	Fix an HQAF $\mathcal{F} = (\mathbf{A}, \mathbf{R}_k, \mathbf{R}_n, \mathbf{R}_d, \mathbf{R}_e, \mathbf{U}, \mathbf{P})$ and an assignment $\|\cdot\|: \mathbf{U} \to [0,1]$. By Definition~\ref{def:cfne_semantics_hqaf}, $\|\cdot\| \in \mathfrak{LS}_{ec_n}^{\mathcal{PL}_{[0,1]}^*}(\mathcal{F})$ if and only if $\left\| ec_n(\mathcal{F}) \right\| = 1$ in $\mathcal{PL}_{[0,1]}^{*, \otimes}$. Since $ec_n(\mathcal{F}) = \bigwedge_{\beta \in \mathbb{U}} \Phi_\beta$, where
	\[
	\Phi_\beta \stackrel{\text{def}}{=} \beta \leftrightarrow \left(
	\begin{aligned}
		&\bigwedge_{\mathbf{k}_{\gamma}^{\beta} \in \mathbf{R}_k} \neg \left( \mathbf{k}_{\gamma}^{\beta} \land \gamma \right) \land \bigwedge_{\mathbf{n}_{\delta}^{\beta} \in \mathbf{R}_n} \neg \left( \mathbf{n}_{\delta}^{\beta} \land \neg \delta \right) \\
		&\land \bigwedge_{\mathbf{d}_{\beta}^{\theta} \in \mathbf{R}_d} \neg \left( \mathbf{d}_{\beta}^{\theta} \land \neg \theta \right) \land \bigvee_{\mathbf{e}_{\varepsilon}^{\beta} \in \mathbf{R}_e} \left( \mathbf{e}_{\varepsilon}^{\beta} \land \varepsilon \right)
	\end{aligned}
	\right),
	\]
	we have $\left\| \bigwedge_{\beta} \Phi_\beta \right\| = 1$ if and only if for every $\beta \in \mathbb{U}$, $\left\| \Phi_\beta \right\| = 1$. In any fuzzy logic with residuated implication, $\|a \leftrightarrow b\| = 1$ if and only if $\|a\| = \|b\|$. Hence for each $\beta$,
	\[
	\left\| \Phi_\beta \right\| = 1 \iff \|\beta\| = \left\|	\begin{aligned}
		&\bigwedge_{\mathbf{k}_{\gamma}^{\beta} \in \mathbf{R}_k} \neg \left( \mathbf{k}_{\gamma}^{\beta} \land \gamma \right) \land \bigwedge_{\mathbf{n}_{\delta}^{\beta} \in \mathbf{R}_n} \neg \left( \mathbf{n}_{\delta}^{\beta} \land \neg \delta \right) \\
		&\land \bigwedge_{\mathbf{d}_{\beta}^{\theta} \in \mathbf{R}_d} \neg \left( \mathbf{d}_{\beta}^{\theta} \land \neg \theta \right) \land \bigvee_{\mathbf{e}_{\varepsilon}^{\beta} \in \mathbf{R}_e} \left( \mathbf{e}_{\varepsilon}^{\beta} \land \varepsilon \right)
	\end{aligned}
	 \right\|.
	\]
	
	Now interpret the connectives in $\mathcal{PL}_{[0,1]}^{*, \otimes}$: $\|a \land b\| = \|a\| \otimes \|b\|$, $\|\neg a\| = N^*(\|a\|)$, and $\|a \vee b\| = N^*(N^*(\|a\|) \otimes N^*(\|b\|))$ (the t-conorm), with finite conjunctions/disjunctions iterated accordingly.
	
	Thus the right-hand side of the equality becomes
	\[
	\begin{aligned}[t]
		&\left( \bigotimes_{\mathbf{k}_{\gamma}^{\beta} \in \mathcal{K}_\beta} N^*\left( \left\| \mathbf{k}_{\gamma}^{\beta} \right\| \otimes \left\| \gamma \right\| \right) \right) \otimes \left( \bigotimes_{\mathbf{n}_{\delta}^{\beta} \in \mathcal{N}_\beta} N^*\left( \left\| \mathbf{n}_{\delta}^{\beta} \right\| \otimes N^*\left(\left\| \delta \right\|\right) \right) \right) \\
		&\otimes \left( \bigotimes_{\mathbf{d}_{\beta}^{\theta} \in \mathcal{D}_\beta} N^*\left( \left\| \mathbf{d}_{\beta}^{\theta} \right\| \otimes N^*\left(\left\| \theta \right\|\right) \right) \right) \otimes \left( N^*\left( \bigotimes_{\mathbf{e}_{\varepsilon}^{\beta} \in \mathcal{E}_\beta} N^*\left( \left\| \mathbf{e}_{\varepsilon}^{\beta} \right\| \otimes \left\| \varepsilon \right\| \right) \right) \right),
	\end{aligned}
	\]
	which is exactly the right-hand side of the fixed-point equation in Definition~\ref{def:cfoe_system_hqaf}. Therefore,
	\[
	\left\| \Phi_\beta \right\| = 1 \iff \|\beta\| = \text{RHS of the CFOE equation}.
	\]
	
	Since this holds for every $\beta \in \mathbb{U}$ and any assignment is trivially consistent among auxiliary elements, we have $\|\cdot\| \vDash_{\mathcal{PL}_{[0,1]}^{*, \otimes}} ec_n(\mathcal{F})$ if and only if $\|\cdot\| \vDash_{eq_{[0,1]}^{*, \otimes}} \mathcal{F}$. Hence the two semantics coincide.
\end{proof}
\subsubsection{Core Properties of CFNE Semantics}\label{4.3.4}

Let $\mathcal{F} = (\mathbf{A}, \mathbf{R}_k, \mathbf{R}_n, \mathbf{R}_d, \mathbf{R}_e, \mathbf{U}, \mathbf{P})$ be an HQAF. For each $\beta \in \mathbb{U}$, let its attacking set, necessary supporting set, deductive supporting set and evidential supporting set be respectively:
\[
\mathcal{K}_\beta = \left\{ \mathbf{k}_{\gamma}^{\beta} \in \mathbf{R}_k \right\},\quad
\mathcal{N}_\beta = \left\{ \mathbf{n}_{\delta}^{\beta} \in \mathbf{R}_n \right\},\quad
\mathcal{D}_\beta = \left\{ \mathbf{d}_{\beta}^{\theta} \in \mathbf{R}_d \right\},\quad
\mathcal{E}_\beta = \left\{ \mathbf{e}_{\varepsilon}^{\beta} \in \mathbf{R}_e \right\}.
\]

According to the CFOE formulation of HQAF, we define the corresponding aggregation function
\[
f_\beta: [0,1]^{m_\beta} \to [0,1]
\]
by
\[
f_\beta = 
\underbrace{\bigotimes_{\mathbf{k}_{\gamma}^{\beta} \in \mathcal{K}_\beta} N^*\left( \left\| \mathbf{k}_{\gamma}^{\beta} \right\| \otimes \left\| \gamma \right\| \right)}_{K(\beta): \text{ attack part}}
\otimes
\underbrace{\bigotimes_{\mathbf{n}_{\delta}^{\beta} \in \mathcal{N}_\beta} N^*\left( \left\| \mathbf{n}_{\delta}^{\beta} \right\| \otimes N^*\left(\left\| \delta \right\|\right) \right)}_{N(\beta): \text{ necessary support part}}
\]
\[
\otimes
\underbrace{\bigotimes_{\mathbf{d}_{\beta}^{\theta} \in \mathcal{D}_\beta} N^*\left( \left\| \mathbf{d}_{\beta}^{\theta} \right\| \otimes N^*\left(\left\| \theta \right\|\right) \right)}_{D(\beta): \text{ deductive support part}}
\otimes
\underbrace{N^*\left( \bigotimes_{\mathbf{e}_{\varepsilon}^{\beta} \in \mathcal{E}_\beta} N^*\left( \left\| \mathbf{e}_{\varepsilon}^{\beta} \right\| \otimes \left\| \varepsilon \right\| \right) \right)}_{E(\beta): \text{ evidential support part}},
\]
where all $\otimes$ denote finite iterations of the continuous t-norm, and $m_\beta$ is the total number of distinct elements (interaction names and source/target elements) relevant to $\beta$. We characterize the core properties of CFNE semantics through the function $f_\beta$, omitting trivial constant assignments for auxiliary elements.

\paragraph{Continuity}
Since both the negation $N^*$ and the t-norm $\otimes$ are continuous by assumption, and $f_\beta$ is built from finite compositions and finite iterations of these operations, $f_\beta$ is continuous on the entire domain $[0,1]^{m_\beta}$.

\paragraph{Commutativity}
The value of $f_\beta$ is invariant under any permutation of elements within each source set, and any permutation of attack groups, necessary support groups, deductive support groups, or evidential support groups. This follows immediately from the commutativity and associativity of the t-norm $\otimes$.

\begin{thm}[Monotonicity of $f_\beta$]\label{thm:monotonicity_fbeta_hqaf}
	The function $f_\beta$ is monotone with respect to the truth value of every element in $\mathbf{U}$, with the following directions:
	\begin{itemize}
		\item Non-increasing with respect to any attack name targeting $\beta$ or any element in the source set of an attack against $\beta$;
		\item Non-decreasing with respect to any necessary support name targeting $\beta$ or any element in the source set of a necessary support for $\beta$;
		\item Non-decreasing with respect to any deductive support name originating from $\beta$ or any element in the target set of a deductive support from $\beta$;
		\item Non-decreasing with respect to any evidential support name targeting $\beta$ or any element in the source set of an evidential support for $\beta$.
	\end{itemize}
\end{thm}

\begin{proof}
	We analyze the four components separately, based on the monotonicity of t-norms and the antitonicity of negation.
	
	\paragraph{Attack part $K(\beta)$}
	Take any attack-related element $u$: either an attack name $\mathbf{k}_{\gamma}^{\beta}$ or a source element $\gamma$ of some attack against $\beta$.
	The inner term $\left\| \mathbf{k}_{\gamma}^{\beta} \right\| \otimes \left\| \gamma \right\|$ is non-decreasing in $u$. Since $N^*$ is antitone, $N^*\left( \left\| \mathbf{k}_{\gamma}^{\beta} \right\| \otimes \left\| \gamma \right\| \right)$ is non-increasing in $u$. The outer t-norm is non-decreasing in each of its factors, so the whole attack part $K(\beta)$ is non-increasing in $u$.
	
	\paragraph{Necessary support part $N(\beta)$}
	Take any necessary support-related element $v$: either a necessary support name $\mathbf{n}_{\delta}^{\beta}$ or a source element $\delta$ of some necessary support for $\beta$.
	The inner term $\left\| \mathbf{n}_{\delta}^{\beta} \right\| \otimes N^*(\left\| \delta \right\|)$ is non-increasing in $v$. Applying the negation $N^*$ reverses the order, so $N^*\left( \left\| \mathbf{n}_{\delta}^{\beta} \right\| \otimes N^*(\left\| \delta \right\|) \right)$ is non-decreasing in $v$. The outer t-norm is non-decreasing in each factor, hence $N(\beta)$ is non-decreasing in $v$.
	
	\paragraph{Deductive support part $D(\beta)$}
	Take any deductive support-related element $w$: either a deductive support name $\mathbf{d}_{\beta}^{\theta}$ or a target element $\theta$ of some deductive support from $\beta$.
	The inner term $\left\| \mathbf{d}_{\beta}^{\theta} \right\| \otimes N^*(\left\| \theta \right\|)$ is non-increasing in $w$. Applying the negation $N^*$ reverses the order, so $N^*\left( \left\| \mathbf{d}_{\beta}^{\theta} \right\| \otimes N^*(\left\| \theta \right\|) \right)$ is non-decreasing in $w$. The outer t-norm is non-decreasing in each factor, hence $D(\beta)$ is non-decreasing in $w$.
	
	\paragraph{Evidential support part $E(\beta)$}
	Take any evidential support-related element $z$: either an evidential support name $\mathbf{e}_{\varepsilon}^{\beta}$ or a source element $\varepsilon$ of some evidential support for $\beta$.
	The inner term $\left\| \mathbf{e}_{\varepsilon}^{\beta} \right\| \otimes \left\| \varepsilon \right\|$ is non-decreasing in $z$. Applying the inner negation $N^*$ yields $N^*\left( \left\| \mathbf{e}_{\varepsilon}^{\beta} \right\| \otimes \left\| \varepsilon \right\| \right)$, which is non-increasing in $z$. The inner t-norm product over all supports is also non-increasing in $z$. Applying the outer negation $N^*$ reverses the order again, hence $E(\beta)$ is non-decreasing in $z$.
	
	\medskip
	Finally, the outer t-norm aggregating $K(\beta), N(\beta), D(\beta), E(\beta)$ is non-decreasing in all four components. Combining the monotonicity directions of each part yields the desired result.
\end{proof}

We first formalize the terminology for interaction validity, then state the four boundary conditions for a given $\beta \in \mathbb{U}$:

\begin{itemize}
	\item An attack $\mathbf{k}_{\gamma}^{\beta}$ is \emph{effective} if $\left\| \mathbf{k}_{\gamma}^{\beta} \right\| = 1$ and $\left\| \gamma \right\| = 1$; it is \emph{ineffective} if $\left\| \mathbf{k}_{\gamma}^{\beta} \right\| = 0$ or $\left\| \gamma \right\| = 0$.
	\item A necessary support $\mathbf{n}_{\delta}^{\beta}$ is \emph{satisfied} if $\left\| \mathbf{n}_{\delta}^{\beta} \right\| = 0$ or $\left\| \delta \right\| = 1$; it is \emph{unsatisfied} if $\left\| \mathbf{n}_{\delta}^{\beta} \right\| = 1$ and $\left\| \delta \right\| = 0$.
	\item A deductive support $\mathbf{d}_{\beta}^{\theta}$ is \emph{satisfied} if $\left\| \mathbf{d}_{\beta}^{\theta} \right\| = 0$ or $\left\| \theta \right\| = 1$; it is \emph{unsatisfied} if $\left\| \mathbf{d}_{\beta}^{\theta} \right\| = 1$ and $\left\| \theta \right\| = 0$.
	\item An evidential support $\mathbf{e}_{\varepsilon}^{\beta}$ is \emph{effective} if $\left\| \mathbf{e}_{\varepsilon}^{\beta} \right\| = 1$ and $\left\| \varepsilon \right\| = 1$; it is \emph{ineffective} if $\left\| \mathbf{e}_{\varepsilon}^{\beta} \right\| = 0$ or $\left\| \varepsilon \right\| = 0$.
\end{itemize}

\begin{itemize}
	\item[(A)] All attacks are ineffective: for every $\mathbf{k}_{\gamma}^{\beta} \in \mathcal{K}_\beta$, $\mathbf{k}_{\gamma}^{\beta}$ is ineffective.
	\item[(N)] All necessary supports are satisfied: for every $\mathbf{n}_{\delta}^{\beta} \in \mathcal{N}_\beta$, $\mathbf{n}_{\delta}^{\beta}$ is satisfied.
	\item[(D)] All deductive supports are satisfied: for every $\mathbf{d}_{\beta}^{\theta} \in \mathcal{D}_\beta$, $\mathbf{d}_{\beta}^{\theta}$ is satisfied.
	\item[(E)] At least one evidential support is effective: there exists $\mathbf{e}_{\varepsilon}^{\beta} \in \mathcal{E}_\beta$ such that $\mathbf{e}_{\varepsilon}^{\beta}$ is effective.
\end{itemize}

\begin{thm}[Boundary Conditions of $f_\beta$]\label{thm:boundary_fbeta_hqaf}
	For every $\beta \in \mathbb{U}$, the following two extremal cases hold:
	\begin{enumerate}
		\item Rejection boundary: if there exists an effective attack, or there exists an unsatisfied necessary support, or there exists an unsatisfied deductive support, or all evidential supports are ineffective, then $f_\beta = 0$.
		\item Acceptance boundary: if all attacks are ineffective, all necessary supports are satisfied, all deductive supports are satisfied, and at least one evidential support is effective, then $f_\beta = 1$.
	\end{enumerate}
\end{thm}

\begin{proof}
	We prove the two cases separately.
	
	\paragraph{Case 1: Rejection boundary}
	We consider four subcases:
	\begin{itemize}
		\item \emph{Effective attack.} Let $\mathbf{k}_{\gamma}^{\beta} \in \mathcal{K}_\beta$ be an effective attack, i.e. $\left\| \mathbf{k}_{\gamma}^{\beta} \right\| = 1$ and $\left\| \gamma \right\| = 1$. Then $\left\| \mathbf{k}_{\gamma}^{\beta} \right\| \otimes \left\| \gamma \right\| = 1$, and thus $N^*(1) = 0$. Hence the corresponding factor in $K(\beta)$ equals $0$. By the annihilation property of t-norms ($0 \otimes a = 0$ for all $a \in [0,1]$), the attack part $K(\beta) = 0$, so $f_\beta = 0 \otimes N(\beta) \otimes D(\beta) \otimes E(\beta) = 0$.
		
		\item \emph{Unsatisfied necessary support.} Let $\mathbf{n}_{\delta}^{\beta} \in \mathcal{N}_\beta$ be an unsatisfied necessary support, i.e. $\left\| \mathbf{n}_{\delta}^{\beta} \right\| = 1$ and $\left\| \delta \right\| = 0$. Then $N^*(\left\| \delta \right\|) = 1$, so $\left\| \mathbf{n}_{\delta}^{\beta} \right\| \otimes N^*(\left\| \delta \right\|) = 1$, and thus $N^*(1) = 0$. Hence the corresponding factor in $N(\beta)$ equals $0$, so $N(\beta) = 0$, and therefore $f_\beta = 0$.
		
		\item \emph{Unsatisfied deductive support.} Let $\mathbf{d}_{\beta}^{\theta} \in \mathcal{D}_\beta$ be an unsatisfied deductive support, i.e. $\left\| \mathbf{d}_{\beta}^{\theta} \right\| = 1$ and $\left\| \theta \right\| = 0$. Then $N^*(\left\| \theta \right\|) = 1$, so $\left\| \mathbf{d}_{\beta}^{\theta} \right\| \otimes N^*(\left\| \theta \right\|) = 1$, and thus $N^*(1) = 0$. Hence the corresponding factor in $D(\beta)$ equals $0$, so $D(\beta) = 0$, and therefore $f_\beta = 0$.
		
		\item \emph{All evidential supports ineffective.} For every $\mathbf{e}_{\varepsilon}^{\beta} \in \mathcal{E}_\beta$, since it is ineffective, either $\left\| \mathbf{e}_{\varepsilon}^{\beta} \right\| = 0$ or $\left\| \varepsilon \right\| = 0$. In either case, $\left\| \mathbf{e}_{\varepsilon}^{\beta} \right\| \otimes \left\| \varepsilon \right\| = 0$, so $N^*\left( \left\| \mathbf{e}_{\varepsilon}^{\beta} \right\| \otimes \left\| \varepsilon \right\| \right) = N^*(0) = 1$. Therefore the inner t-norm product over all supports equals $\bigotimes 1 = 1$. Applying the outer negation yields $E(\beta) = N^*(1) = 0$, so $f_\beta = K(\beta) \otimes N(\beta) \otimes D(\beta) \otimes 0 = 0$.
	\end{itemize}
	In all four subcases, $f_\beta = 0$.
	
	\paragraph{Case 2: Acceptance boundary}
	If all attacks are ineffective, then for every $\mathbf{k}_{\gamma}^{\beta} \in \mathcal{K}_\beta$, $\left\| \mathbf{k}_{\gamma}^{\beta} \right\| \otimes \left\| \gamma \right\| = 0$, and $N^*(0) = 1$. Hence every factor in $K(\beta)$ equals $1$, so $K(\beta) = \bigotimes 1 = 1$.
	
	If all necessary supports are satisfied, then for every $\mathbf{n}_{\delta}^{\beta} \in \mathcal{N}_\beta$, $\left\| \mathbf{n}_{\delta}^{\beta} \right\| \otimes N^*(\left\| \delta \right\|) = 0$, and $N^*(0) = 1$. Hence every factor in $N(\beta)$ equals $1$, so $N(\beta) = 1$.
	
	If all deductive supports are satisfied, then for every $\mathbf{d}_{\beta}^{\theta} \in \mathcal{D}_\beta$, $\left\| \mathbf{d}_{\beta}^{\theta} \right\| \otimes N^*(\left\| \theta \right\|) = 0$, and $N^*(0) = 1$. Hence every factor in $D(\beta)$ equals $1$, so $D(\beta) = 1$.
	
	If at least one evidential support is effective, let $\mathbf{e}_{\varepsilon}^{\beta} \in \mathcal{E}_\beta$ be such that $\left\| \mathbf{e}_{\varepsilon}^{\beta} \right\| = 1$ and $\left\| \varepsilon \right\| = 1$. Then $\left\| \mathbf{e}_{\varepsilon}^{\beta} \right\| \otimes \left\| \varepsilon \right\| = 1$, so $N^*(1) = 0$. The inner t-norm product over all supports contains at least one factor $0$, hence the whole product equals $0$. Consequently, $E(\beta) = N^*(0) = 1$.
	
	Combining the four parts, $f_\beta = 1 \otimes 1 \otimes 1 \otimes 1 = 1$.
\end{proof}
\begin{thm}[CFOE Equation Trifurcation]\label{thm:cfoe_trifurcation_hqaf}
	Assume that the t-norm $\otimes$ employed in the function $f_\beta$ is zero-divisor-free (i.e., $x \otimes y = 0 \implies x = 0$ or $y = 0$). For every $\beta \in \mathbb{U}$, the following equivalences hold:
	\begin{enumerate}
		\item $f_\beta = 1 \iff \text{(A) and (N) and (D) and (E)}$.
		\item $f_\beta = 0 \iff$ there exists an effective attack, or there exists an unsatisfied necessary support, or there exists an unsatisfied deductive support, or all evidential supports are ineffective.
		\item $0 < f_\beta < 1 \iff$ neither of the above holds.
	\end{enumerate}
\end{thm}

\begin{proof}
	Write $f_\beta = K(\beta) \otimes N(\beta) \otimes D(\beta) \otimes E(\beta)$, where
	\begin{align*}
		K(\beta) &= \bigotimes_{\mathbf{k}_{\gamma}^{\beta} \in \mathcal{K}_\beta} N^*\left( \left\| \mathbf{k}_{\gamma}^{\beta} \right\| \otimes \left\| \gamma \right\| \right), \\
		N(\beta) &= \bigotimes_{\mathbf{n}_{\delta}^{\beta} \in \mathcal{N}_\beta} N^*\left( \left\| \mathbf{n}_{\delta}^{\beta} \right\| \otimes N^*\left(\left\| \delta \right\|\right) \right), \\
		D(\beta) &= \bigotimes_{\mathbf{d}_{\beta}^{\theta} \in \mathcal{D}_\beta} N^*\left( \left\| \mathbf{d}_{\beta}^{\theta} \right\| \otimes N^*\left(\left\| \theta \right\|\right) \right), \\
		E(\beta) &= N^*\left( \bigotimes_{\mathbf{e}_{\varepsilon}^{\beta} \in \mathcal{E}_\beta} N^*\left( \left\| \mathbf{e}_{\varepsilon}^{\beta} \right\| \otimes \left\| \varepsilon \right\| \right) \right).
	\end{align*}
	We use the facts: $N^*(0) = 1$, $N^*(1) = 0$, and $\otimes$ is monotone, associative, commutative, has identity $1$ and zero $0$, and is zero-divisor-free.
	
	\textit{1. $f_\beta = 1 \iff$ (A) and (N) and (D) and (E).}
	
	Since $1 \otimes x = x$ and $x \otimes y = 1$ implies $x = y = 1$, we have $f_\beta = 1$ if and only if $K(\beta) = N(\beta) = D(\beta) = E(\beta) = 1$.
	
	Now $K(\beta) = 1$ iff every factor in the product is $1$, i.e., for every attack $\mathbf{k}_{\gamma}^{\beta}$, $N^*\left( \left\| \mathbf{k}_{\gamma}^{\beta} \right\| \otimes \left\| \gamma \right\| \right) = 1$, which is equivalent to $\left\| \mathbf{k}_{\gamma}^{\beta} \right\| \otimes \left\| \gamma \right\| = 0$. By zero-divisor-freeness, this holds iff either $\left\| \mathbf{k}_{\gamma}^{\beta} \right\| = 0$ or $\left\| \gamma \right\| = 0$, i.e., condition (A).
	
	$N(\beta) = 1$ iff every factor in the product is $1$, i.e., for every necessary support $\mathbf{n}_{\delta}^{\beta}$, $N^*\left( \left\| \mathbf{n}_{\delta}^{\beta} \right\| \otimes N^*(\left\| \delta \right\|) \right) = 1$, which is equivalent to $\left\| \mathbf{n}_{\delta}^{\beta} \right\| \otimes N^*(\left\| \delta \right\|) = 0$. By zero-divisor-freeness, this holds iff either $\left\| \mathbf{n}_{\delta}^{\beta} \right\| = 0$ or $N^*(\left\| \delta \right\|) = 0$, i.e., $\left\| \delta \right\| = 1$. This is exactly condition (N).
	
	$D(\beta) = 1$ iff every factor in the product is $1$, i.e., for every deductive support $\mathbf{d}_{\beta}^{\theta}$, $N^*\left( \left\| \mathbf{d}_{\beta}^{\theta} \right\| \otimes N^*(\left\| \theta \right\|) \right) = 1$, which is equivalent to $\left\| \mathbf{d}_{\beta}^{\theta} \right\| \otimes N^*(\left\| \theta \right\|) = 0$. By zero-divisor-freeness, this holds iff either $\left\| \mathbf{d}_{\beta}^{\theta} \right\| = 0$ or $N^*(\left\| \theta \right\|) = 0$, i.e., $\left\| \theta \right\| = 1$. This is exactly condition (D).
	
	$E(\beta) = 1$ iff $N^*\left( \bigotimes_{\mathbf{e}_{\varepsilon}^{\beta} \in \mathcal{E}_\beta} N^*\left( \left\| \mathbf{e}_{\varepsilon}^{\beta} \right\| \otimes \left\| \varepsilon \right\| \right) \right) = 1$, i.e., the inner t-norm product equals $0$. By zero-divisor-freeness, a product of numbers in $[0,1]$ is $0$ iff at least one factor is $0$. Thus $E(\beta) = 1$ iff there exists an evidential support $\mathbf{e}_{\varepsilon}^{\beta}$ such that $N^*\left( \left\| \mathbf{e}_{\varepsilon}^{\beta} \right\| \otimes \left\| \varepsilon \right\| \right) = 0$, i.e., $\left\| \mathbf{e}_{\varepsilon}^{\beta} \right\| \otimes \left\| \varepsilon \right\| = 1$, which is equivalent to $\left\| \mathbf{e}_{\varepsilon}^{\beta} \right\| = 1$ and $\left\| \varepsilon \right\| = 1$, i.e., condition (E).
	
	Hence $f_\beta = 1$ if and only if (A), (N), (D) and (E) all hold.
	
	\textit{2. $f_\beta = 0 \iff$ the rejection condition holds.}
	
	Since $x \otimes y = 0$ iff $x = 0$ or $y = 0$ (zero-divisor-free property), $f_\beta = 0$ iff $K(\beta) = 0$ or $N(\beta) = 0$ or $D(\beta) = 0$ or $E(\beta) = 0$.
	
	$K(\beta) = 0$ iff there exists an attack with $N^*\left( \left\| \mathbf{k}_{\gamma}^{\beta} \right\| \otimes \left\| \gamma \right\| \right) = 0$, i.e., $\left\| \mathbf{k}_{\gamma}^{\beta} \right\| \otimes \left\| \gamma \right\| = 1$, which is equivalent to $\left\| \mathbf{k}_{\gamma}^{\beta} \right\| = 1$ and $\left\| \gamma \right\| = 1$, i.e., an effective attack exists.
	
	$N(\beta) = 0$ iff there exists a necessary support with $N^*\left( \left\| \mathbf{n}_{\delta}^{\beta} \right\| \otimes N^*(\left\| \delta \right\|) \right) = 0$, i.e., $\left\| \mathbf{n}_{\delta}^{\beta} \right\| \otimes N^*(\left\| \delta \right\|) = 1$, which is equivalent to $\left\| \mathbf{n}_{\delta}^{\beta} \right\| = 1$ and $N^*(\left\| \delta \right\|) = 1$, i.e., $\left\| \delta \right\| = 0$. This means an unsatisfied necessary support exists.
	
	$D(\beta) = 0$ iff there exists a deductive support with $N^*\left( \left\| \mathbf{d}_{\beta}^{\theta} \right\| \otimes N^*(\left\| \theta \right\|) \right) = 0$, i.e., $\left\| \mathbf{d}_{\beta}^{\theta} \right\| \otimes N^*(\left\| \theta \right\|) = 1$, which is equivalent to $\left\| \mathbf{d}_{\beta}^{\theta} \right\| = 1$ and $N^*(\left\| \theta \right\|) = 1$, i.e., $\left\| \theta \right\| = 0$. This means an unsatisfied deductive support exists.
	
	$E(\beta) = 0$ iff $N^*\left( \bigotimes_{\mathbf{e}_{\varepsilon}^{\beta} \in \mathcal{E}_\beta} N^*\left( \left\| \mathbf{e}_{\varepsilon}^{\beta} \right\| \otimes \left\| \varepsilon \right\| \right) \right) = 0$, i.e., the inner t-norm product equals $1$. This is equivalent to every factor in the product being $1$, i.e., for every evidential support $\mathbf{e}_{\varepsilon}^{\beta}$, $N^*\left( \left\| \mathbf{e}_{\varepsilon}^{\beta} \right\| \otimes \left\| \varepsilon \right\| \right) = 1$, equivalently $\left\| \mathbf{e}_{\varepsilon}^{\beta} \right\| \otimes \left\| \varepsilon \right\| = 0$. By zero-divisor-freeness, this holds iff either $\left\| \mathbf{e}_{\varepsilon}^{\beta} \right\| = 0$ or $\left\| \varepsilon \right\| = 0$, meaning all evidential supports are ineffective.
	
	Hence $f_\beta = 0$ if and only if the rejection condition holds.
	
	\textit{3. The intermediate case.}
	
	The values of $f_\beta$ lie in $[0,1]$. The above two cases characterize exactly when $f_\beta = 0$ and when $f_\beta = 1$. Therefore, if neither condition holds, we have $0 < f_\beta < 1$. Conversely, if $0 < f_\beta < 1$, then neither $f_\beta = 0$ nor $f_\beta = 1$ holds, so neither the acceptance nor the rejection condition is satisfied. This completes the proof.
\end{proof}

\begin{thm}[Existence of Solutions]\label{thm:solution_existence_hqaf}
	For every HQAF $\mathcal{F}=(\mathbf{A},\mathbf{R}_k,\mathbf{R}_n,\mathbf{R}_d$, $\mathbf{R}_e,\mathbf{U},\mathbf{P})$, the CFOE system $eq_{[0,1]}^{*,\otimes}$ has at least one solution. Consequently, the CFNE semantic model set $\mathfrak{LS}_{ec_n}^{\mathcal{PL}_{[0,1]}^{*,\otimes}}(\mathcal{F})$ is non-empty for every $\mathcal{F}$.
\end{thm}

\begin{proof}
	Let $n=|\mathbf{U}|$ and enumerate all elements of $\mathbf{U}$ as $\beta_1,\beta_2,\ldots,\beta_n$. Every assignment $\|\cdot\|:\mathbf{U}\to[0,1]$ corresponds uniquely to a vector $x=(x_1,x_2,\ldots,x_n)\in[0,1]^n$ with $x_i=\|\beta_i\|$.
	
	The CFOE system $eq_{[0,1]}^{*,\otimes}$ consists of two types of equations. For every auxiliary element $\beta_i$, the truth value is fixed by a constant equation
	\[
	\|\beta_i\|=c_i,
	\]
	where $c_i=0$ if $\beta_i=\bot$, and $c_i=1$ if $\beta_i=\top$ or $\beta_i \in (\mathbf{R}_k \cup \mathbf{R}_n \cup \mathbf{R}_d \cup \mathbf{R}_e) \setminus (\mathbb{R}_k \cup \mathbb{R}_n \cup \mathbb{R}_d \cup \mathbb{R}_e)$. For every non-auxiliary element $\beta_i$, the truth value is determined by the fixed-point equation
	\[
	\|\beta_i\|=f_{\beta_i}\bigl(\|x_{i_1}\|,\ldots,\|x_{i_{m}}\|\bigr),
	\]
	where $f_{\beta_i}$ is the aggregation function built in Definition \ref{def:cfoe_system_hqaf}.
	
	Define the vector-valued map
	\[
	F:[0,1]^n\to[0,1]^n,\qquad F(x)=(F_1(x),F_2(x),\ldots,F_n(x)),
	\]
	by
	\[
	F_i(x)=
	\begin{cases}
		c_i, & \text{if } \beta_i \text{ is an auxiliary element},\\[1mm]
		f_{\beta_i}(x), & \text{if } \beta_i \text{ is a non-auxiliary element}.
	\end{cases}
	\]
	We verify the hypotheses of Brouwer's fixed-point theorem.
	
	\smallskip
	\noindent\textbf{1. Domain.} The unit cube $[0,1]^n$ is non-empty, compact, and convex.
	
	\smallskip
	\noindent\textbf{2. Continuity.} For auxiliary elements, $F_i$ is a constant function, hence continuous. For non-auxiliary elements, $f_{\beta_i}$ is continuous on $[0,1]^n$ because it is a finite composition of the continuous negation $N^*$ and the continuous t-norm $\otimes$. Thus every component $F_i$ is continuous, so $F$ is continuous.
	
	\smallskip
	\noindent\textbf{3. Self-mapping.} For any $x\in[0,1]^n$, each $F_i(x)$ lies in $[0,1]$: this is immediate for auxiliary elements since $c_i\in\{0,1\}$, and for non-auxiliary elements since $f_{\beta_i}$ is closed on $[0,1]$. Hence $F(x)\in[0,1]^n$.
	
	By Brouwer's fixed-point theorem, there exists $x^*\in[0,1]^n$ such that $F(x^*)=x^*$. Unfolding the definition of $F$, for every auxiliary element $\beta_i$ we have $x_i^*=c_i$, so the constant equations are satisfied. For every non-auxiliary element $\beta_i$ we have $x_i^*=f_{\beta_i}(x^*)$, so the corresponding fixed-point equation is satisfied. Therefore the assignment $\|\cdot\|^*$ induced by $x^*$ is a solution of the CFOE system $eq_{[0,1]}^{*,\otimes}$.
	
	Finally, by the equivalence of CFOE and CFNE semantics, this solution is also a model of the CFNE semantics. Hence the model set $\mathfrak{LS}_{ec_n}^{\mathcal{PL}_{[0,1]}^{*,\otimes}}(\mathcal{F})$ is non-empty.
\end{proof}
\begin{cor}\label{cor2}
	For any HQAF, each of the fuzzy encoded semantics $\mathfrak{LS}_{ec_n}^{\mathcal{PL}_{[0,1]}^G}$, $\mathfrak{LS}_{ec_n}^{\mathcal{PL}_{[0,1]}^P}$, and $\mathfrak{LS}_{ec_n}^{\mathcal{PL}_{[0,1]}^L}$ has at least one model.
\end{cor}

\begin{proof}
	Immediate from Theorem \ref{thm:solution_existence_hqaf} by taking the t-norm $\otimes$ to be the Gödel, Product, or Łukasiewicz t-norm, respectively.
\end{proof}
\subsection{Relationships between CFNE Semantics and 3-Valued Semantics}

We first define the ternarization function that maps fuzzy assignments to three-valued assignments. Let $\mathcal{LAB}_3$ denote the set of all 3-valued labellings of HQAF on $\{0,1,\frac{1}{2}\}$.

\begin{defn}[Ternarization for HQAF]\label{def:ternarization_hqaf}
	The ternarization function $T_3 : \mathcal{LAB} \to \mathcal{LAB}_3$ maps a fuzzy assignment $\|\cdot\|$ to a $3$-valued assignment $\|\cdot\|_3 = T_3(\|\cdot\|)$ defined by
	\[
	\|x\|_3 =
	\begin{cases}
		1, & \|x\| = 1, \\
		0, & \|x\| = 0, \\
		\frac{1}{2}, & \text{otherwise}.
	\end{cases}
	\]
\end{defn}

\subsubsection{Semantic Correspondence under Zero‑Divisor‑Free t‑Norms}

Recall that the three-valued equational system $eq_3^{HQ}$ (Definition~\ref{def:3val_equational_hqaf}) is given by
\[
\|\beta\|_3 = \min\left\{ K_3(\beta),\, N_3(\beta),\, D_3(\beta),\, E_3(\beta) \right\},
\]
with
\begin{align*}
	K_3(\beta) &= \min_{\mathbf{k}_{\gamma}^{\beta} \in \mathbf{R}_k} \max\left\{ 1-\left\| \mathbf{k}_{\gamma}^{\beta} \right\|_3,\, 1-\left\| \gamma \right\|_3 \right\}, \\
	N_3(\beta) &= \min_{\mathbf{n}_{\delta}^{\beta} \in \mathbf{R}_n} \max\left\{ 1-\left\| \mathbf{n}_{\delta}^{\beta} \right\|_3,\, \left\| \delta \right\|_3 \right\}, \\
	D_3(\beta) &= \min_{\mathbf{d}_{\beta}^{\theta} \in \mathbf{R}_d} \max\left\{ 1-\left\| \mathbf{d}_{\beta}^{\theta} \right\|_3,\, \left\| \theta \right\|_3 \right\}, \\
	E_3(\beta) &= \max_{\mathbf{e}_{\varepsilon}^{\beta} \in \mathbf{R}_e} \min\left\{ \left\| \mathbf{e}_{\varepsilon}^{\beta} \right\|_3,\, \left\| \varepsilon \right\|_3 \right\},
\end{align*}
where all operations are evaluated over the three-valued domain $\left\{0, \frac{1}{2}, 1\right\}$.

\begin{thm}[Ternarisation of CFNE Models]\label{thm:ternarization_cfne_hqaf}
	Let the CFNE semantics be induced by the normal encoding and a propositional logic $\mathcal{PL}_{[0,1]}$ equipped with a continuous negation and a continuous zero‑divisor‑free t‑norm. For any HQAF $\mathcal{F}$ and any fuzzy assignment $\|\cdot\|$,
	\[
	\|\cdot\| \vDash_{CFNE} \mathcal{F} \implies T_3(\|\cdot\|) \text{ is a solution of } eq_3^{HQ}.
	\]
	Consequently, by Corollary~\ref{thm:3val_equivalence_hqaf}, $T_3(\|\cdot\|)$ is an adjacent complete labelling of $\mathcal{F}$.
\end{thm}

\begin{proof}
	Assume $\|\cdot\| \vDash_{CFNE} \mathcal{F}$. Then for every $\beta \in \mathbb{U}$, we have $\|\beta\| = f_\beta$ with $f_\beta$ as in Definition~\ref{def:cfoe_system_hqaf}. Let $\|\cdot\|_3 = T_3(\|\cdot\|)$. We show that for every $\beta \in \mathbb{U}$,
	\[
	\|\beta\|_3 = \min\left\{ K_3(\beta),\, N_3(\beta),\, D_3(\beta),\, E_3(\beta) \right\}.
	\]
	We apply Theorem~\ref{thm:cfoe_trifurcation_hqaf} to the value $\|\beta\|$ and consider three cases.
	
	\paragraph{Case 1: $\|\beta\| = 1$}
	By Theorem~\ref{thm:cfoe_trifurcation_hqaf}, conditions (A), (N), (D) and (E) all hold under $\|\cdot\|$.
	Condition (A) means that for every attack $\mathbf{k}_{\gamma}^{\beta}$, either $\left\| \mathbf{k}_{\gamma}^{\beta} \right\| = 0$ or $\left\| \gamma \right\| = 0$. In ternarized terms, this implies $\left\| \mathbf{k}_{\gamma}^{\beta} \right\|_3 = 0$ or $\left\| \gamma \right\|_3 = 0$, so $\max\left\{1-\left\| \mathbf{k}_{\gamma}^{\beta} \right\|_3, 1-\left\| \gamma \right\|_3\right\} = 1$ for every attack. Hence $K_3(\beta) = 1$.
	Similarly, condition (N) gives $N_3(\beta) = 1$, condition (D) gives $D_3(\beta) = 1$, and condition (E) gives $E_3(\beta) = 1$.
	Thus $\min\left\{ K_3(\beta), N_3(\beta), D_3(\beta), E_3(\beta) \right\} = 1 = \|\beta\|_3$.
	
	\paragraph{Case 2: $\|\beta\| = 0$}
	By Theorem~\ref{thm:cfoe_trifurcation_hqaf}, at least one of the rejection conditions holds under $\|\cdot\|$.
	If there exists an effective attack, then for that attack $\left\| \mathbf{k}_{\gamma}^{\beta} \right\| = 1$ and $\left\| \gamma \right\| = 1$, so $\left\| \mathbf{k}_{\gamma}^{\beta} \right\|_3 = 1$ and $\left\| \gamma \right\|_3 = 1$, hence $\max\left\{1-\left\| \mathbf{k}_{\gamma}^{\beta} \right\|_3, 1-\left\| \gamma \right\|_3\right\} = 0$, and therefore $K_3(\beta) = 0$.
	If there exists an unsatisfied necessary support, then $N_3(\beta) = 0$ by the same reasoning.
	Similarly, if there exists an unsatisfied deductive support, then $D_3(\beta) = 0$.
	If all evidential supports are ineffective, then for every evidential support $\left\| \mathbf{e}_{\varepsilon}^{\beta} \right\| = 0$ or $\left\| \varepsilon \right\| = 0$, so $\min\left\{\left\| \mathbf{e}_{\varepsilon}^{\beta} \right\|_3, \left\| \varepsilon \right\|_3\right\} = 0$ for every support, hence $E_3(\beta) = 0$.
	In all subcases, $\min\left\{ K_3(\beta), N_3(\beta), D_3(\beta), E_3(\beta) \right\} = 0 = \|\beta\|_3$.
	
	\paragraph{Case 3: $0 < \|\beta\| < 1$}
	Then neither the acceptance condition nor the rejection condition holds under $\|\cdot\|$, by Theorem~\ref{thm:cfoe_trifurcation_hqaf}. Therefore, neither condition holds under $\|\cdot\|_3$ either. Thus, from Definition \ref{adjcom} and Corollary \ref{thm:3val_equivalence_hqaf}, the right-hand side of the 3-valued equation for $\beta$ evaluates to $\frac{1}{2}$ under this $\|\cdot\|_3$. Hence the equality holds, since $\|\beta\|_3 = \frac{1}{2}$ as well.
	
	\medskip
	Thus in all cases, the equation holds for every $\beta \in \mathbb{U}$. The auxiliary elements have fixed values satisfying the system trivially. Therefore $T_3(\|\cdot\|)$ is a solution of $eq_3^{HQ}$. By Corollary~\ref{thm:3val_equivalence_hqaf}, it is an adjacent complete labelling of $\mathcal{F}$.
\end{proof}

\subsubsection{Semantic Correspondence under $\frac{1}{2}$-Idempotent t-Norms}

Let $\mathcal{F} = (\mathbf{A}, \mathbf{R}_k, \mathbf{R}_n, \mathbf{R}_d, \mathbf{R}_e, \mathbf{U}, \mathbf{P})$ be an HQAF. For each $\beta \in \mathbb{U}$, let $\mathcal{K}_\beta$ be the set of attacks against $\beta$, $\mathcal{N}_\beta$ the set of necessary supports for $\beta$, $\mathcal{D}_\beta$ the set of deductive supports from $\beta$, and $\mathcal{E}_\beta$ the set of evidential supports for $\beta$. Let $\odot$ be a continuous $\frac{1}{2}$-idempotent t-norm (i.e., $\frac{1}{2} \odot \frac{1}{2} = \frac{1}{2}$) and let $N$ be the standard negation $N(x) = 1-x$. We consider assignments $\|\cdot\|: \mathbf{U} \to \{0, \frac{1}{2}, 1\}$.

For $\beta \in \mathbb{U}$, define the following quantities:
\begin{align*}
	K_\beta &\stackrel{\text{def}}{=} \bigodot_{\mathbf{k}_{\gamma}^{\beta} \in \mathcal{K}_\beta} N\left( \left\| \mathbf{k}_{\gamma}^{\beta} \right\| \odot \left\| \gamma \right\| \right), \\
	N_\beta &\stackrel{\text{def}}{=} \bigodot_{\mathbf{n}_{\delta}^{\beta} \in \mathcal{N}_\beta} N\left( \left\| \mathbf{n}_{\delta}^{\beta} \right\| \odot N\left( \left\| \delta \right\| \right) \right), \\
	D_\beta &\stackrel{\text{def}}{=} \bigodot_{\mathbf{d}_{\beta}^{\theta} \in \mathcal{D}_\beta} N\left( \left\| \mathbf{d}_{\beta}^{\theta} \right\| \odot N\left( \left\| \theta \right\| \right) \right), \\
	E_\beta &\stackrel{\text{def}}{=} N\left( \bigodot_{\mathbf{e}_{\varepsilon}^{\beta} \in \mathcal{E}_\beta} N\left( \left\| \mathbf{e}_{\varepsilon}^{\beta} \right\| \odot \left\| \varepsilon \right\| \right) \right).
\end{align*}

Ignoring the fixed auxiliary‑element assignments (which are trivially satisfied), the CFNE semantics induced by the normal encoding and a $\mathcal{PL}_{[0,1]}$ equipped with the standard negation and the continuous $\frac{1}{2}$-idempotent t-norm is equivalent to the equational system
\[
\|\beta\| = K_\beta \odot N_\beta \odot D_\beta \odot E_\beta \quad \text{for every } \beta \in \mathbb{U}. \tag{$\ast$}
\]

Employing the notations introduced above, the adjacent complete semantics is given by:
\[
\|\beta\| =
\begin{cases}
	1, & \text{iff } 
	\begin{aligned}[t]
		&(\forall \mathbf{k}_{\gamma}^{\beta} \in \mathcal{K}_\beta,\; \|\gamma\| = 0 \text{ or } \|\mathbf{k}_{\gamma}^{\beta}\| = 0) \text{ and } \\
		&(\forall \mathbf{n}_{\delta}^{\beta} \in \mathcal{N}_\beta,\; \|\delta\| = 1 \text{ or } \|\mathbf{n}_{\delta}^{\beta}\| = 0) \text{ and } \\
		&(\forall \mathbf{d}_{\beta}^{\theta} \in \mathcal{D}_\beta,\; \|\theta\| = 1 \text{ or } \|\mathbf{d}_{\beta}^{\theta}\| = 0) \text{ and } \\
		&(\exists \mathbf{e}_{\varepsilon}^{\beta} \in \mathcal{E}_\beta,\; \|\varepsilon\| = 1 \text{ and } \|\mathbf{e}_{\varepsilon}^{\beta}\| = 1)
	\end{aligned}
	\\[2ex]
	0, & \text{iff } 
	\begin{aligned}[t]
		&(\exists \mathbf{k}_{\gamma}^{\beta} \in \mathcal{K}_\beta,\; \|\gamma\| = 1 \text{ and } \|\mathbf{k}_{\gamma}^{\beta}\| = 1) \text{ or } \\
		&(\exists \mathbf{n}_{\delta}^{\beta} \in \mathcal{N}_\beta,\; \|\delta\| = 0 \text{ and } \|\mathbf{n}_{\delta}^{\beta}\| = 1) \text{ or } \\
		&(\exists \mathbf{d}_{\beta}^{\theta} \in \mathcal{D}_\beta,\; \|\theta\| = 0 \text{ and } \|\mathbf{d}_{\beta}^{\theta}\| = 1) \text{ or } \\
		&(\forall \mathbf{e}_{\varepsilon}^{\beta} \in \mathcal{E}_\beta,\; \|\varepsilon\| = 0 \text{ or } \|\mathbf{e}_{\varepsilon}^{\beta}\| = 0)
	\end{aligned}
	\\[2ex]
	\frac{1}{2}, & \text{otherwise.}
\end{cases}\tag{$\ast\ast$}
\]

We now prove that every 3‑valued assignment satisfying three equivalences ($\ast\ast$) also satisfies Equation ($\ast$).

\begin{thm}\label{thm:3val_is_cfne_hqaf}
	Let the CFNE semantics be induced by the normal encoding and a propositional logic $\mathcal{PL}_{[0,1]}$ equipped with a continuous negation and a continuous $\frac{1}{2}$-idempotent t‑norm. For every HQAF $\mathcal{F}$ and every $\|\cdot\|: \mathbf{U} \to \{0, \frac{1}{2}, 1\}$,
	\[
	\|\cdot\| \text{ is an adjacent complete labelling} \implies \|\cdot\| \vDash_{CFNE} \mathcal{F}.
	\]
\end{thm}

\begin{proof}
	Fix $\beta \in \mathbb{U}$. We show the equivalence $\|\beta\| = K_\beta \odot N_\beta \odot D_\beta \odot E_\beta$ under any adjacent complete labelling $\|\cdot\|$ by proving the three cases. We rely on the following fact: since all values lie in $\{0, \frac{1}{2}, 1\}$ and $\odot$ is $\frac{1}{2}$-idempotent, the $\odot$-product of any finite number of such values is:
	\begin{itemize}
		\item $1$ if and only if every factor is $1$;
		\item $0$ if and only if at least one factor is $0$;
		\item $\frac{1}{2}$ otherwise (no factor equals $0$, but at least one factor equals $\frac{1}{2}$).
	\end{itemize}
	Consequently, for any $x,y$, we have $x \odot y = 0 \iff x=0$ or $y=0$, and $x \odot y = 1 \iff x=1$ and $y=1$.
	
	\paragraph{Case 1: $\|\beta\| = 0$}
	We prove the equivalence:
	\[
	\|\beta\| = 0 \iff K_\beta \odot N_\beta \odot D_\beta \odot E_\beta = 0.
	\]
	
	First, assume $\|\beta\| = 0$. By ($\ast\ast$), either (i) there exists an effective attack, or (ii) there exists an unsatisfied necessary support, or (iii) there exists an unsatisfied deductive support, or (iv) all evidential supports are ineffective.
	
	(i) If there exists $\mathbf{k}_{\gamma}^{\beta}$ with $\left\| \mathbf{k}_{\gamma}^{\beta} \right\| = 1$ and $\|\gamma\| = 1$, then
	\[
	\left\| \mathbf{k}_{\gamma}^{\beta} \right\| \odot \left\| \gamma \right\| = 1,
	\]
	so $N\left( \left\| \mathbf{k}_{\gamma}^{\beta} \right\| \odot \left\| \gamma \right\| \right) = 0$. Hence $K_\beta$ contains a factor $0$, so $K_\beta = 0$, and therefore $K_\beta \odot N_\beta \odot D_\beta \odot E_\beta = 0$.
	
	(ii) If there exists an unsatisfied necessary support $\mathbf{n}_{\delta}^{\beta}$, then $\left\| \mathbf{n}_{\delta}^{\beta} \right\| = 1$ and $\|\delta\| = 0$, so $N(\|\delta\|) = 1$, and
	\[
	\left\| \mathbf{n}_{\delta}^{\beta} \right\| \odot N\left( \left\| \delta \right\| \right) = 1,
	\]
	so $N\left( \left\| \mathbf{n}_{\delta}^{\beta} \right\| \odot N\left( \left\| \delta \right\| \right) \right) = 0$. Hence $N_\beta = 0$, and the whole product is $0$.
	
	(iii) If there exists an unsatisfied deductive support $\mathbf{d}_{\beta}^{\theta}$, then similarly $D_\beta = 0$, and the product is $0$.
	
	(iv) If all evidential supports are ineffective, then for each $\mathbf{e}_{\varepsilon}^{\beta}$,
	\[
	\left\| \mathbf{e}_{\varepsilon}^{\beta} \right\| \odot \left\| \varepsilon \right\| = 0,
	\]
	so $N\left( \left\| \mathbf{e}_{\varepsilon}^{\beta} \right\| \odot \left\| \varepsilon \right\| \right) = 1$ for every support. Thus the inner product $\bigodot_{\mathbf{e} \in \mathcal{E}_\beta} N(\cdots)$ is a product of $1$’s, hence equals $1$. Therefore $E_\beta = N(1) = 0$, and the whole product is $0$.
	
	Conversely, assume $K_\beta \odot N_\beta \odot D_\beta \odot E_\beta = 0$. Since $x \odot y = 0$ iff $x=0$ or $y=0$, we have $K_\beta = 0$ or $N_\beta = 0$ or $D_\beta = 0$ or $E_\beta = 0$.
	
	If $K_\beta = 0$, then there exists an attack such that
	\[
	N\left( \left\| \mathbf{k}_{\gamma}^{\beta} \right\| \odot \left\| \gamma \right\| \right) = 0,
	\]
	which implies $\left\| \mathbf{k}_{\gamma}^{\beta} \right\| \odot \left\| \gamma \right\| = 1$. This is equivalent to $\left\| \mathbf{k}_{\gamma}^{\beta} \right\| = 1$ and $\|\gamma\| = 1$, i.e., an effective attack exists.
	
	If $N_\beta = 0$, then there exists a necessary support such that
	\[
	N\left( \left\| \mathbf{n}_{\delta}^{\beta} \right\| \odot N\left( \left\| \delta \right\| \right) \right) = 0,
	\]
	so $\left\| \mathbf{n}_{\delta}^{\beta} \right\| \odot N\left( \left\| \delta \right\| \right) = 1$. This means $\left\| \mathbf{n}_{\delta}^{\beta} \right\| = 1$ and $N(\|\delta\|) = 1$, i.e., $\|\delta\| = 0$, so an unsatisfied necessary support exists.
	
	If $D_\beta = 0$, similarly, an unsatisfied deductive support exists.
	
	If $E_\beta = 0$, then
	\[
	\bigodot_{\mathbf{e}_{\varepsilon}^{\beta} \in \mathcal{E}_\beta} N\left( \left\| \mathbf{e}_{\varepsilon}^{\beta} \right\| \odot \left\| \varepsilon \right\| \right) = 1,
	\]
	which implies that every factor is $1$, i.e., for each support $\mathbf{e}_{\varepsilon}^{\beta}$, $N\left( \left\| \mathbf{e}_{\varepsilon}^{\beta} \right\| \odot \left\| \varepsilon \right\| \right) = 1$, so $\left\| \mathbf{e}_{\varepsilon}^{\beta} \right\| \odot \left\| \varepsilon \right\| = 0$. This means either $\left\| \mathbf{e}_{\varepsilon}^{\beta} \right\| = 0$ or $\|\varepsilon\| = 0$, i.e., every support is ineffective.
	
	Thus in every case, the condition in ($\ast\ast$) for $\|\beta\| = 0$ holds, so $\|\beta\| = 0$.
	
	Therefore, $\|\beta\| = 0 \iff K_\beta \odot N_\beta \odot D_\beta \odot E_\beta = 0$.
	
	\paragraph{Case 2: $\|\beta\| = 1$}
	We prove the equivalence:
	\[
	\|\beta\| = 1 \iff K_\beta \odot N_\beta \odot D_\beta \odot E_\beta = 1.
	\]
	
	Assume $\|\beta\| = 1$. By ($\ast\ast$), every attack is blocked, every necessary support is satisfied, every deductive support is satisfied, and there exists an effective evidential support.
	
	If every attack is blocked, then for each $\mathbf{k}_{\gamma}^{\beta}$,
	\[
	\left\| \mathbf{k}_{\gamma}^{\beta} \right\| \odot \left\| \gamma \right\| = 0,
	\]
	so $N(0) = 1$. Hence every factor in $K_\beta$ is $1$, so $K_\beta = 1$.
	
	If every necessary support is satisfied, then for each $\mathbf{n}_{\delta}^{\beta}$,
	\[
	\left\| \mathbf{n}_{\delta}^{\beta} \right\| \odot N\left( \left\| \delta \right\| \right) = 0,
	\]
	so $N(0) = 1$. Hence every factor in $N_\beta$ is $1$, so $N_\beta = 1$.
	
	If every deductive support is satisfied, then similarly $D_\beta = 1$.
	
	If there exists an effective evidential support, then for that $\mathbf{e}_{\varepsilon}^{\beta}$,
	\[
	\left\| \mathbf{e}_{\varepsilon}^{\beta} \right\| \odot \left\| \varepsilon \right\| = 1,
	\]
	so $N(1) = 0$. Thus the inner product over supports contains a factor $0$, hence equals $0$. Therefore $E_\beta = N(0) = 1$.
	
	Consequently, $K_\beta \odot N_\beta \odot D_\beta \odot E_\beta = 1 \odot 1 \odot 1 \odot 1 = 1$.
	
	Conversely, assume $K_\beta \odot N_\beta \odot D_\beta \odot E_\beta = 1$. Since $x \odot y = 1$ iff $x=1$ and $y=1$, we have $K_\beta = 1$, $N_\beta = 1$, $D_\beta = 1$ and $E_\beta = 1$.
	
	$K_\beta = 1$ means that every factor in the product is $1$, i.e., for every attack, $N\left( \left\| \mathbf{k}_{\gamma}^{\beta} \right\| \odot \left\| \gamma \right\| \right) = 1$, so $\left\| \mathbf{k}_{\gamma}^{\beta} \right\| \odot \left\| \gamma \right\| = 0$. This is equivalent to either $\left\| \mathbf{k}_{\gamma}^{\beta} \right\| = 0$ or $\|\gamma\| = 0$, i.e., every attack is blocked.
	
	$N_\beta = 1$ means that every factor is $1$, i.e., for every necessary support, $N\left( \left\| \mathbf{n}_{\delta}^{\beta} \right\| \odot N(\|\delta\|) \right) = 1$, so $\left\| \mathbf{n}_{\delta}^{\beta} \right\| \odot N(\|\delta\|) = 0$. This is equivalent to either $\left\| \mathbf{n}_{\delta}^{\beta} \right\| = 0$ or $\|\delta\| = 1$, i.e., every necessary support is satisfied.
	
	$D_\beta = 1$ similarly implies every deductive support is satisfied.
	
	$E_\beta = 1$ means that
	\[
	\bigodot_{\mathbf{e} \in \mathcal{E}_\beta} N\left( \left\| \mathbf{e}_{\varepsilon}^{\beta} \right\| \odot \left\| \varepsilon \right\| \right) = 0.
	\]
	Since the product is $0$, at least one factor is $0$, i.e., there exists a support $\mathbf{e}_{\varepsilon}^{\beta}$ such that $N\left( \left\| \mathbf{e}_{\varepsilon}^{\beta} \right\| \odot \left\| \varepsilon \right\| \right) = 0$, so $\left\| \mathbf{e}_{\varepsilon}^{\beta} \right\| \odot \left\| \varepsilon \right\| = 1$. This means $\left\| \mathbf{e}_{\varepsilon}^{\beta} \right\| = 1$ and $\|\varepsilon\| = 1$, i.e., an effective support exists.
	
	Thus the conditions in ($\ast\ast$) for $\|\beta\| = 1$ hold, so $\|\beta\| = 1$.
	
	Therefore, $\|\beta\| = 1 \iff K_\beta \odot N_\beta \odot D_\beta \odot E_\beta = 1$.
	
	\paragraph{Case 3: $\|\beta\| = \frac{1}{2}$}
	By ($\ast\ast$), $\|\beta\|$ is neither $0$ nor $1$. From the equivalences proven in Cases 1 and 2, we have:
	\[
	\|\beta\| = 0 \iff K_\beta \odot N_\beta \odot D_\beta \odot E_\beta = 0, \qquad
	\|\beta\| = 1 \iff K_\beta \odot N_\beta \odot D_\beta \odot E_\beta = 1.
	\]
	Hence $K_\beta \odot N_\beta \odot D_\beta \odot E_\beta$ is neither $0$ nor $1$. Since the $\odot$-product of values from $\{0, \frac{1}{2}, 1\}$ also lies in that set (by $\frac{1}{2}$-idempotency), the only remaining value is $\frac{1}{2}$. Thus
	\[
	K_\beta \odot N_\beta \odot D_\beta \odot E_\beta = \frac{1}{2} = \|\beta\|.
	\]
	
	In all three cases, we have established $\|\beta\| = K_\beta \odot N_\beta \odot D_\beta \odot E_\beta$ under any adjacent complete labelling $\|\cdot\|$. Since $\beta$ was arbitrary, Equation ($\ast$) holds for every element of $\mathbb{U}$. The auxiliary elements $\bot$, $\top$ and the added attacks/supports have fixed values that trivially satisfy the system. Hence $\|\cdot\| \vDash_{CFNE} \mathcal{F}$.
\end{proof}

\begin{thm}\label{thm:3val_cfne_equivalence_hqaf}
	Let the CFNE semantics be induced by the normal encoding and a $\mathcal{PL}_{[0,1]}$ equipped with a continuous $\frac{1}{2}$-idempotent t-norm and a standard negation. For any HQAF $\mathcal{F}$,
	\[
	\left\{ \|\cdot\| \mid \|\cdot\| \text{ is an adjacent complete labelling of } \mathcal{F} \right\}
	=
	\left\{ T_3(\|\cdot\|) \mid \|\cdot\| \vDash_{CFNE} \mathcal{F} \right\}.
	\]
\end{thm}

\begin{proof}
	A $\frac{1}{2}$-idempotent t-norm must also be zero-divisor-free. Therefore, from Theorem~\ref{thm:ternarization_cfne_hqaf}, if $\|\cdot\|$ is a model of $\mathcal{F}$ under the given CFNE semantics, then $T_3(\|\cdot\|)$ is an adjacent complete labelling of $\mathcal{F}$. Thus,
	\[
	\left\{ T_3(\|\cdot\|) \mid \|\cdot\| \vDash_{CFNE} \mathcal{F} \right\}
	\subseteq
	\left\{ \|\cdot\| \mid \|\cdot\| \text{ is an adjacent complete labelling of } \mathcal{F} \right\}.
	\]
	
	From Theorem~\ref{thm:3val_is_cfne_hqaf}, if a 3-valued labelling $\|\cdot\|$ is an adjacent complete labelling of $\mathcal{F}$, then it is a model of $\mathcal{F}$ under the given CFNE semantics. Since $\|\cdot\| = T_3(\|\cdot\|)$ for this 3-valued labelling, we have
	\[
	\left\{ \|\cdot\| \mid \|\cdot\| \text{ is an adjacent complete labelling of } \mathcal{F} \right\}
	\subseteq
	\left\{ T_3(\|\cdot\|) \mid \|\cdot\| \vDash_{CFNE} \mathcal{F} \right\}.
	\]
	
	This completes the proof.
\end{proof}
\subsection{Key Instances of Fuzzy Encoded Semantics}

Based on the unified CFOE/CFNE framework established in Section 4.3, we present three canonical instances induced by standard continuous t-norms: Gödel, Product and Łukasiewicz. Each instance corresponds to a specific aggregation pattern of attacks and supports, and the semantic equivalence between equational and encoded forms follows directly from Theorem \ref{thm:cfoe_cfne_equivalence_hqaf}. We also discuss their core algebraic properties and intuitive interpretations.

\subsubsection{Gödel Fuzzy Encoded Semantics}

The CFOE system induced by Gödel t-norm $\otimes_G = \min$ is the most natural extension of the discrete max-min equational semantics to the continuous $[0,1]$ domain.

\begin{thm}[Gödel-type CFOE System]\label{thm:goedel_cfoe_hqaf}
	Let $\mathcal{F} = (\mathbf{A}, \mathbf{R}_k, \mathbf{R}_n, \mathbf{R}_d, \mathbf{R}_e, \mathbf{U}, \mathbf{P})$ be an HQAF and $\|\cdot\|: \mathbf{U} \to [0,1]$ be an assignment. If the negation is the standard negation $N^*(x) = 1-x$ and the t-norm is the Gödel t-norm $\otimes_G = \min$, then the Gödel-type CFOE system $eq_{[0,1]}^G$ is given by the auxiliary-element conditions
	\[
	\|\bot\| = 0,\quad \|\top\| = 1,\quad \|\alpha\| = 1 \quad \big(\forall \alpha \in (\mathbf{R}_k \cup \mathbf{R}_n \cup \mathbf{R}_d \cup \mathbf{R}_e) \setminus (\mathbb{R}_k \cup \mathbb{R}_n \cup \mathbb{R}_d \cup \mathbb{R}_e)\big)
	\]
	and for each $\beta \in \mathbb{U}$, the fixed-point equation:
	\[
	\|\beta\| = \min\left\{
	\underbrace{\min_{\mathbf{k}_{\gamma}^{\beta} \in \mathbf{R}_k} \max\left\{ 1-\left\| \mathbf{k}_{\gamma}^{\beta} \right\|,\, 1-\left\| \gamma \right\| \right\}}_{K_G(\beta): \text{ attack part}},
	\right.
	\]
	\[
	\quad\quad\quad
	\underbrace{\min_{\mathbf{n}_{\delta}^{\beta} \in \mathbf{R}_n} \max\left\{ 1-\left\| \mathbf{n}_{\delta}^{\beta} \right\|,\, \left\| \delta \right\| \right\}}_{N_G(\beta): \text{ necessary support part}},
	\]
	\[
	\quad\quad\quad
	\left.
	\underbrace{\min_{\mathbf{d}_{\beta}^{\theta} \in \mathbf{R}_d} \max\left\{ 1-\left\| \mathbf{d}_{\beta}^{\theta} \right\|,\, \left\| \theta \right\| \right\}}_{D_G(\beta): \text{ deductive support part}},
	\underbrace{\max_{\mathbf{e}_{\varepsilon}^{\beta} \in \mathbf{R}_e} \min\left\{ \left\| \mathbf{e}_{\varepsilon}^{\beta} \right\|,\, \left\| \varepsilon \right\| \right\}}_{E_G(\beta): \text{ evidential support part}}
	\right\}.
	\]
\end{thm}

\begin{proof}
	The result follows immediately from Definition \ref{def:cfoe_system_hqaf} by specifying the continuous negation as the standard negation $N^*(x) = 1-x$ and the Gödel t-norm $\otimes_G = \min$.
\end{proof}

Intuitively, the attack part $K_G(\beta)$ measures the degree to which all attacks against $\beta$ fail. For each attack, its failure degree equals the maximum failure degree among the attack name and the source element of that attack; the overall attack failure degree is the minimum of these values, i.e., determined by the single most effective attack against $\beta$.

The necessary support part $N_G(\beta)$ measures the degree to which all necessary supports targeting $\beta$ are satisfied. For each necessary support, its satisfaction degree equals the maximum of the invalidity degree of the support name and the validity degree of its source argument; the overall satisfaction degree of necessary supports is the minimum of these values, reflecting the intuition that the acceptability of $\beta$ is constrained by the least satisfied necessary support.

The deductive support part $D_G(\beta)$ measures the degree to which all deductive supports originating from $\beta$ are satisfied. For each deductive support, its satisfaction degree equals the maximum of the invalidity degree of the support name and the validity degree of its target argument; the overall satisfaction degree of deductive supports is the minimum of these values, capturing the requirement that \(\beta\) can be fully accepted only if every deductive inference originating from \(\beta\) is satisfied.

The evidential support part $E_G(\beta)$ measures the degree to which at least one evidential support for $\beta$ succeeds. For each evidential support, its success degree equals the minimum success degree among the support name and its source argument; the overall evidential support success degree is the maximum of these values, i.e., governed by the strongest evidential support available for $\beta$.

The final acceptability degree of $\beta$ is the truth degree of the conjunction of all four conditions, aggregated by the Gödel t-norm (minimum). This formulation naturally extends the discrete max-min equational semantics to the continuous unit interval, preserving the classical weakest-link principle for conjunctive constraints and the strongest-link principle for disjunctive evidential support.

\begin{thm}[Equivalence of Gödel CFOE and CFNE Semantics]
	For every HQAF $\mathcal{F}$, the Gödel-type continuous fuzzy operator-based equational semantics coincides with the Gödel fuzzy normal encoded semantics, i.e.,
	\[
	\mathfrak{LS}_{Eq_{[0,1]}^G}(\mathcal{F}) = \mathfrak{LS}_{ec_n}^{\mathcal{PL}_{[0,1]}^G}(\mathcal{F}).
	\]
\end{thm}

\begin{proof}
	The result follows immediately from Theorem \ref{thm:cfoe_cfne_equivalence_hqaf} by specifying the continuous negation as the standard negation $N^*(x)=1-x$ and the t-norm as the Gödel t-norm $\otimes_G = \min$.
\end{proof}

As a special case of the general CFNE framework, the Gödel fuzzy encoded semantics inherits all core algebraic properties established in Section \ref{4.3.4}. We restate them below for the Gödel instance for completeness:
\begin{enumerate}
	\item \textbf{Continuity}. The aggregation function $f_\beta$ for each element $\beta$ is continuous on $[0,1]^{m_\beta}$.
	\item \textbf{Boundary conditions}. If there exists an effective attack, an unsatisfied necessary support, an unsatisfied deductive support, or all evidential supports are ineffective, then $\|\beta\| = 0$; if all attacks are ineffective, all necessary supports are satisfied, all deductive supports are satisfied, and at least one evidential support is effective, then $\|\beta\| = 1$.
	\item \textbf{Monotonicity}. The truth degree $\|\beta\|$ is non-increasing with respect to the truth value of any attack-related element, and non-decreasing with respect to the truth value of any element belonging to necessary supports, deductive supports, or evidential supports.
	\item \textbf{Existence of solutions}. The Gödel-type CFOE system for any HQAF has at least one solution, or equivalently, the model set of the Gödel CFNE semantics is non-empty.
\end{enumerate}

From Theorems \ref{thm:ternarization_cfne_hqaf}, \ref{thm:3val_is_cfne_hqaf} and \ref{thm:3val_cfne_equivalence_hqaf}, we obtain the following corollaries characterizing the relationship between the Gödel CFNE semantics and the adjacent complete labelling semantics.

\begin{cor}
	For any HQAF $F$ and any fuzzy assignment $\|\cdot\|: U \to [0,1]$,
	\[
	\|\cdot\| \vDash_{\mathfrak{LS}_{ec_n}^{\mathcal{PL}_{[0,1]}^G}} F \implies T_3(\|\cdot\|) \text{ is an adjacent complete labelling of } F.
	\]
\end{cor}

\begin{proof}
	Since the Gödel t-norm is zero-divisor-free and the standard negation is continuous in $\mathcal{PL}_{[0,1]}^G$, the statement follows directly from Theorem \ref{thm:ternarization_cfne_hqaf}.
\end{proof}

\begin{cor}
	For every HQAF $F$ and every 3-valued assignment $\|\cdot\|: U \to \{0, \frac{1}{2}, 1\}$,
	\[
	\|\cdot\| \text{ is an adjacent complete labelling} \implies \|\cdot\| \vDash_{\mathfrak{LS}_{ec_n}^{\mathcal{PL}_{[0,1]}^G}} F.
	\]
\end{cor}

\begin{proof}
	The Gödel t-norm is $\frac{1}{2}$-idempotent, and the standard negation is well-defined in $\mathcal{PL}_{[0,1]}^G$. Therefore, the statement follows directly from Theorem \ref{thm:3val_is_cfne_hqaf}.
\end{proof}

\begin{cor}\label{cor5}
	For any HQAF $F$,
	\[
	\left\{\|\cdot\| \mid \|\cdot\| \vDash_{\mathfrak{LS}_{ac}} \mathcal{F}\right\} = \left\{ T_3(\|\cdot\|) \mid \|\cdot\| \vDash_{\mathfrak{LS}_{ec_n}^{\mathcal{PL}_{[0,1]}^G}} \mathcal{F} \right\}.
	\]
\end{cor}

\begin{proof}
	The standard negation is continuous, and the Gödel t-norm is continuous, $\frac{1}{2}$-idempotent in $\mathcal{PL}_{[0,1]}^G$. Hence, the equality follows directly from Theorem \ref{thm:3val_cfne_equivalence_hqaf}.
\end{proof}
\begin{cor}[Existence of Models for Adjacent Complete Labelling Semantics]
	\label{cor:existence-ac}
	For any HQAF $\mathcal{F}$, the adjacent complete labelling semantics has at least one model, i.e.,
	\[
	\mathfrak{L}\mathfrak{S}_{ac}(\mathcal{F}) \neq \emptyset.
	\]
\end{cor}

\begin{proof}
	By Corollary~\ref{cor2} and Corollary~\ref{cor5}, we have $\mathfrak{L}\mathfrak{S}_{ac}(\mathcal{F}) \neq \emptyset$.
\end{proof}
\subsubsection{Product Fuzzy Encoded Semantics}

The Product t-norm $\otimes_P(x, y) = x \cdot y$ corresponds to probabilistic independent aggregation: it naturally models the joint effectiveness of multiple independent attacks and supports. This semantics is suitable for argumentation scenarios where different attack sources and support sources are mutually independent.

\begin{thm}[Product-type CFOE System]
	Let $\mathcal{F}=(\mathbf{A},\mathbf{R}_k,\mathbf{R}_n,\mathbf{R}_d,\mathbf{R}_e,\mathbf{U}$, $\mathbf{P})$ be an HQAF and $\|\cdot\|:\mathbf{U}\to[0,1]$ be an assignment. If the negation is the standard negation $N^*(x)=1-x$ and the t-norm is the Product t-norm $\otimes_P(x,y)=x\cdot y$, then the Product-type CFOE system $eq_{[0,1]}^P$ is given by the auxiliary-element conditions
	\[
	\|\bot\|=0,\quad \|\top\|=1,\quad \|\alpha\|=1\quad \bigl(\forall \alpha\in(\mathbf{R}_k\cup\mathbf{R}_n\cup\mathbf{R}_d\cup\mathbf{R}_e)\setminus(\mathbb{R}_k\cup\mathbb{R}_n\cup\mathbb{R}_d\cup\mathbb{R}_e)\bigr)
	\]
	and for each $\beta\in\mathbb{U}$, the fixed-point equation:
	\[
	\begin{aligned}
		\|\beta\| = & \Biggl( \prod_{\mathbf{k}_{\gamma}^{\beta}\in \mathcal{K}_{\beta}} \bigl(1-\|\mathbf{k}_{\gamma}^{\beta}\|\cdot\|\gamma\|\bigr) \Biggr) \\
		&\cdot \Biggl( \prod_{\mathbf{n}_{\delta}^{\beta}\in \mathcal{N}_{\beta}} \bigl(1-\|\mathbf{n}_{\delta}^{\beta}\|\cdot(1-\|\delta\|)\bigr) \Biggr) \\
		&\cdot \Biggl( \prod_{\mathbf{d}_{\beta}^{\theta}\in \mathcal{D}_{\beta}} \bigl(1-\|\mathbf{d}_{\beta}^{\theta}\|\cdot(1-\|\theta\|)\bigr) \Biggr) \\
		&\cdot \Biggl( 1 - \prod_{\mathbf{e}_{\varepsilon}^{\beta}\in \mathcal{E}_{\beta}} \bigl(1-\|\mathbf{e}_{\varepsilon}^{\beta}\|\cdot\|\varepsilon\|\bigr) \Biggr),
	\end{aligned}
	\]
	where
	\[
	\mathcal{K}_{\beta}=\{\mathbf{k}_{\gamma}^{\beta}\in \mathbf{R}_k\},\quad
	\mathcal{N}_{\beta}=\{\mathbf{n}_{\delta}^{\beta}\in \mathbf{R}_n\},\quad
	\mathcal{D}_{\beta}=\{\mathbf{d}_{\beta}^{\theta}\in \mathbf{R}_d\},\quad
	\mathcal{E}_{\beta}=\{\mathbf{e}_{\varepsilon}^{\beta}\in \mathbf{R}_e\}.
	\]
\end{thm}

\begin{proof}
	The result follows immediately from Definition~\ref{def:cfoe_system_hqaf} by specifying the continuous negation as the standard negation $N^*(x)=1-x$ and the t-norm as the Product t-norm $\otimes_P(x,y)=x\cdot y$.
\end{proof}

Intuitively, the attack part aggregates the joint failure probability of all independent attacks: the success probability of a single attack is the product of the validity of the attack name and the acceptability of its source element; its failure probability is 1 minus this value; the joint failure probability of all attacks is the product of individual failure probabilities.

The necessary support part aggregates the joint satisfaction probability of all independent necessary supports: the violation probability of a single necessary support is the product of the validity of the support name and the invalidity of its source element; its satisfaction probability is 1 minus this value; the joint satisfaction probability is the product of individual satisfaction probabilities.

The deductive support part aggregates the joint satisfaction probability of all independent deductive supports: the violation probability of a single deductive support is the product of the validity of the support name and the invalidity of its target element; its satisfaction probability is 1 minus this value; the joint satisfaction probability is the product of individual satisfaction probabilities.

The evidential support part aggregates the probability that at least one independent evidential support succeeds: the success probability of a single evidential support is the product of the validity of the support name and the acceptability of its source element; the joint failure probability of all evidential supports is the product of individual failure probabilities; the probability that at least one support succeeds is 1 minus this value.

The final acceptability degree of $\beta$ is the joint probability that all attacks fail, all necessary supports are satisfied, all deductive supports are satisfied, and at least one evidential support succeeds.

\begin{thm}[Equivalence of Product CFOE and CFNE Semantics]
	For every HQAF $\mathcal{F}$, the Product-type continuous fuzzy operator-based equational semantics coincides with the Product fuzzy normal encoded semantics, i.e.,
	\[
	\mathfrak{LS}_{Eq_{[0,1]}^P}(\mathcal{F}) = \mathfrak{LS}_{ec_n}^{\mathcal{PL}_{[0,1]}^P}(\mathcal{F}).
	\]
\end{thm}

\begin{proof}
	The result follows immediately from Theorem \ref{thm:cfoe_cfne_equivalence_hqaf} by specifying the continuous negation as the standard negation $N^*(x)=1-x$ and the t-norm as the Product t-norm $\otimes_P(x, y) = x \cdot y$.
\end{proof}

The Product fuzzy encoded semantics also inherits all general algebraic properties from the CFNE framework: continuity, commutativity, boundary conditions and monotonicity all hold, and the existence of solutions is guaranteed by Brouwer’s fixed-point theorem. The following corollaries present the relationships between the Product CFNE semantics and the adjacent complete labelling semantics.

\begin{cor}
	For any HQAF $F$ and any fuzzy assignment $\|\cdot\|: U \to [0,1]$,
	\[
	\|\cdot\| \vDash_{\mathfrak{LS}_{ec_n}^{\mathcal{PL}_{[0,1]}^P}} F \implies T_3(\|\cdot\|) \text{ is an adjacent complete labelling of } F.
	\]
\end{cor}

\begin{proof}
	Since the Product t-norm is zero-divisor-free and the standard negation is continuous in $\mathcal{PL}_{[0,1]}^P$, the statement follows directly from Theorem \ref{thm:ternarization_cfne_hqaf}.
\end{proof}

\begin{cor}
	For any HQAF $F$,
	\[
	\left\{ T_3(\|\cdot\|) \mid \|\cdot\| \vDash_{\mathfrak{LS}_{ec_n}^{\mathcal{PL}_{[0,1]}^P}} \mathcal{F} \right\} \subseteq \left\{\|\cdot\| \mid \|\cdot\| \vDash_{\mathfrak{LS}_{ac}} \mathcal{F}\right\}.
	\]
\end{cor}

\begin{proof}
	The result follows immediately from the preceding corollary.
\end{proof}

Unlike the idempotent Gödel semantics, the Product t-norm is non-idempotent: multiple independent attacks gradually accumulate and weaken the acceptability of the target, rather than being determined solely by the strongest attack. Similarly, multiple independent supports gradually enhance the acceptability. This cumulative feature makes Product semantics more suitable for modeling argumentative reasoning with additive strength.

\subsubsection{Łukasiewicz Fuzzy Encoded Semantics}

The Łukasiewicz t-norm $\otimes_L(x, y) = \max\{0, x+y-1\}$ corresponds to threshold-based additive aggregation. It models mutually exclusive or competitive attacks and supports: the total strength of attacks and supports accumulates linearly, and exerts a full effect only when the total exceeds a certain threshold. This makes it suitable for argumentation scenarios where different attacks and supports compete additively.

\begin{thm}[Łukasiewicz-type CFOE System]
	Let $\mathcal{F}=(\mathbf{A},\mathbf{R}_k,\mathbf{R}_n,\mathbf{R}_d$, $\mathbf{R}_e,\mathbf{U},\mathbf{P})$ be an HQAF and $\|\cdot\|:\mathbf{U}\to[0,1]$ be an assignment. If the negation is the standard negation $N^*(x)=1-x$ and the t-norm is the Łukasiewicz t-norm $\otimes_L(x,y)=\max\{0,x+y-1\}$, then the Łukasiewicz-type CFOE system $eq_{[0,1]}^L$ is given by the auxiliary-element conditions
	\[
	\|\bot\|=0,\quad \|\top\|=1,\quad \|\alpha\|=1\quad \bigl(\forall \alpha\in(\mathbf{R}_k\cup\mathbf{R}_n\cup\mathbf{R}_d\cup\mathbf{R}_e)\setminus(\mathbb{R}_k\cup\mathbb{R}_n\cup\mathbb{R}_d\cup\mathbb{R}_e)\bigr)
	\]
	and for each $\beta\in\mathbb{U}$, the fixed-point equation:
	\[
	\|\beta\| = \max\Bigl\{0,\; K_L(\beta) + N_L(\beta) + D_L(\beta) + E_L(\beta) - 3 \Bigr\},
	\]
	where the four semantic quantities are defined as follows:
	\[
	\begin{aligned}
		K_L(\beta) &= \max\left\{0,\; \sum_{\mathbf{k}_{\gamma}^{\beta}\in \mathcal{K}_{\beta}} \min\bigl\{1,\; 2 - \|\mathbf{k}_{\gamma}^{\beta}\| - \|\gamma\|\bigr\} - (|\mathcal{K}_{\beta}| - 1) \right\}, \\[2mm]
		N_L(\beta) &= \max\left\{0,\; \sum_{\mathbf{n}_{\delta}^{\beta}\in \mathcal{N}_{\beta}} \min\bigl\{1,\; 1 - \|\mathbf{n}_{\delta}^{\beta}\| + \|\delta\|\bigr\} - (|\mathcal{N}_{\beta}| - 1) \right\}, \\[2mm]
		D_L(\beta) &= \max\left\{0,\; \sum_{\mathbf{d}_{\beta}^{\theta}\in \mathcal{D}_{\beta}} \min\bigl\{1,\; 1 - \|\mathbf{d}_{\beta}^{\theta}\| + \|\theta\|\bigr\} - (|\mathcal{D}_{\beta}| - 1) \right\}, \\[2mm]
		E_L(\beta) &= \min\left\{1,\; |\mathcal{E}_{\beta}| - \sum_{\mathbf{e}_{\varepsilon}^{\beta}\in \mathcal{E}_{\beta}} \min\bigl\{1,\; 2 - \|\mathbf{e}_{\varepsilon}^{\beta}\| - \|\varepsilon\|\bigr\} \right\}.
	\end{aligned}
	\]
	Here $\mathcal{K}_{\beta}=\{\mathbf{k}_{\gamma}^{\beta}\in \mathbf{R}_k\}$, $\mathcal{N}_{\beta}=\{\mathbf{n}_{\delta}^{\beta}\in \mathbf{R}_n\}$, $\mathcal{D}_{\beta}=\{\mathbf{d}_{\beta}^{\theta}\in \mathbf{R}_d\}$, and $\mathcal{E}_{\beta}=\{\mathbf{e}_{\varepsilon}^{\beta}\in \mathbf{R}_e\}$.
\end{thm}

\begin{proof}
	The result follows immediately from Definition~\ref{def:cfoe_system_hqaf} by specifying the continuous negation as the standard negation $N^*(x)=1-x$ and the t-norm as the Łukasiewicz t-norm $\otimes_L(x,y)=\max\{0,x+y-1\}$. For a finite family $(x_i)_{i=1}^n$ in $[0,1]$, the iterated Łukasiewicz t-norm is
	\[
	\bigotimes_{i=1}^n x_i = \max\left\{0,\; \sum_{i=1}^n x_i - (n-1)\right\}.
	\]
	Applying these identities to the attack, necessary support, deductive support, and evidential support parts in the general CFOE equation yields the stated expressions for $K_L(\beta)$, $N_L(\beta)$, $D_L(\beta)$, and $E_L(\beta)$. The final aggregation over the four parts is again an iterated Łukasiewicz t-norm, giving
	\begin{align*}
			\|\beta\| &= K_L(\beta) \otimes_L N_L(\beta) \otimes_L D_L(\beta) \otimes_L E_L(\beta)\\
		&= \max\bigl\{0,\; K_L(\beta) + N_L(\beta) + D_L(\beta) + E_L(\beta) - 3\bigr\}.
	\end{align*}
	The auxiliary-element conditions are identical to those in Definition~\ref{def:cfoe_system_hqaf}.
\end{proof}

Intuitively, the Łukasiewicz semantics treats each interaction as contributing a certain amount of “strength” to a linear accumulator, and the final acceptability is determined by whether the total accumulated strength exceeds a threshold. 

The attack part $K_L(\beta)$ aggregates the degree to which all attacks against $\beta$ fail. For each attack $\mathbf{k}_{\gamma}^{\beta}$, the quantity $\min\{1, 2 - \|\mathbf{k}_{\gamma}^{\beta}\| - \|\gamma\|\}$ measures the failure degree of that attack: it equals $1$ when the attack name and its source element are jointly weak enough (i.e., their sum does not exceed $1$), and decreases linearly as their combined strength increases. The iterated Łukasiewicz t-norm over these failure degrees yields a value that is $1$ only if the total “failure strength” is large enough to compensate for the number of attacks; otherwise, the attacks collectively reduce the acceptability. 

The necessary support part $N_L(\beta)$ aggregates the degree to which all necessary supports for $\beta$ are satisfied. For each necessary support $\mathbf{n}_{\delta}^{\beta}$, the quantity $\min\{1, 1 - \|\mathbf{n}_{\delta}^{\beta}\| + \|\delta\|\}$ measures its satisfaction degree: it equals $1$ when the support name is invalid or its source element is fully acceptable, and decreases linearly as the support name becomes more valid and its source element becomes less acceptable. The iterated Łukasiewicz t-norm over these satisfaction degrees captures the requirement that all necessary supports must be jointly satisfied to a sufficient total degree.

The deductive support part $D_L(\beta)$ is analogous to the necessary support part, but it concerns deductive supports originating from $\beta$. For each deductive support $\mathbf{d}_{\beta}^{\theta}$, the quantity $\min\{1, 1 - \|\mathbf{d}_{\beta}^{\theta}\| + \|\theta\|\}$ measures its satisfaction degree, reflecting the requirement that the target element $\theta$ must be acceptable whenever the deductive support is valid.

The evidential support part $E_L(\beta)$ aggregates the degree to which at least one evidential support for $\beta$ succeeds. For each evidential support $\mathbf{e}_{\varepsilon}^{\beta}$, the quantity $\min\{1, 2 - \|\mathbf{e}_{\varepsilon}^{\beta}\| - \|\varepsilon\|\}$ measures its failure degree: it equals $1$ when the support name and its source element are jointly weak, and decreases as their combined strength increases. The inner Łukasiewicz t-norm over these failure degrees yields a value that is $0$ only if at least one evidential support has enough combined strength to succeed; otherwise, it is positive. The outer standard negation then converts this into a success degree: $E_L(\beta) = 1$ if at least one evidential support succeeds with sufficient strength, and $E_L(\beta) = 0$ if all evidential supports fail.

Finally, the acceptability degree of $\beta$ is obtained by combining the four parts using the Łukasiewicz t-norm, which acts as a threshold-based conjunction: the sum $K_L(\beta) + N_L(\beta) + D_L(\beta) + E_L(\beta)$ must exceed $3$ for $\beta$ to receive a positive acceptability degree. This reflects a competitive additive aggregation where the four conditions jointly determine the outcome, and a deficit in any one part can be compensated by a surplus in another, but only up to a certain threshold. 

The Łukasiewicz-type CFOE system is equivalent to the corresponding Łukasiewicz fuzzy normal encoded semantics:
\[
\mathfrak{LS}_{Eq_{[0,1]}^L}(\mathcal{F}) = \mathfrak{LS}_{ec_n}^{\mathcal{PL}_{[0,1]}^L}(\mathcal{F}),
\]
as guaranteed by the general equivalence theorem (Theorem~\ref{thm:cfoe_cfne_equivalence_hqaf}). Since the expanded form above is already explicit, and all core algebraic properties (continuity, monotonicity, boundary conditions, and existence of solutions) are special cases of the general CFNE framework, we omit a separate derivation of those properties here.

\section{Related Work}
\label{sec:related}

\paragraph{Abstract argumentation and bipolar frameworks}
Dung's AF \cite{Dung1995} models acceptability via attacks. BAF adds a support relation and combine it with attacks into complex attacks \cite{Cayrol2005,Amgoud2008,Cayrol2013}. However, BAF presents different interpretations of the support relation. Specialized interpretations include necessary support \cite{Nouioua2011,Nouioua2013}, deductive support \cite{Boella2010}, and evidential support \cite{Oren2008,Oren2010}. Comparative studies and translations between these interpretations are given in \cite{Cayrol2013,Polberg2014}. These works are foundational, but they typically consider at most one support type at a time. HQAF instead allows necessary, deductive, and evidential supports to coexist with attacks, and preserves their typical semantic behaviours.

\paragraph{Higher-order and recursive interactions}
Higher-order attacks have been studied for preferences and for defeating attacks without defeating their sources \cite{Barringer2005,Modgil2009}. AFRA and RAF generalize this to attacks targeting attacks \cite{Baroni2011,Cayrol2017}. ASAF \cite{Cohen2015,Gottifredi2018} combines recursive attacks with necessary support; RAFN \cite{Cayrol2018a} handles higher-order necessary supports; REBAF \cite{Cayrol2018b} handles higher-order evidential supports. Logical encodings for RAFN and REBAF are developed in \cite{Cayrol2020,Lagasquie-Schiex2021a,Lagasquie-Schiex2021b,Lagasquie-Schiex2023}. These frameworks are bipolar: they pair attacks with one support type. Moreover, support sources are typically sets of arguments. HQAF generalizes this by allowing attacks and all three support types to act as sources and targets, while retaining a uniform higher-order treatment.

\paragraph{Multipolar frameworks and encoded semantics}
Gabbay \cite{Gabbay2016} introduces multipolar and tripolar argumentation networks with extension-based labelling semantics, but the framework is not higher-order and lacks encoded semantics. Tang \cite{Tang2025} proposes HAFS, a higher-order framework with supports that allows attacks and supports as sources and targets, and encodes it into propositional logic. A related evidential higher-order set framework EHSAF is studied in \cite{Tang2026}. These works are close in spirit, but HAFS focuses on the necessary support relation, and EHSAF only employs the evidential support relation. HQAF fills the gap by providing a quadripolar framework where attacks, necessary supports, deductive supports, and evidential supports coexist and interact at higher order. When an​ HQAF is restricted to the case of​ a framework with a single support relation, the encoded semantics is identical to​ that of a​n HAFS or an​ EHSAF.

\paragraph{Logical encodings and equational semantics}
Dung's AF can be translated into logic programs \cite{Dung1995}, and propositional encodings for standard semantics are studied in \cite{Besnard2004}. Gabbay \cite{Gabbay2011,gabbay2012equational} introduces equational semantics over numerical values, later extended to bipolar and tripolar networks \cite{Gabbay2016} and to merging \cite{Gabbay2014}. For higher-order bipolar frameworks, first-order encodings are given in \cite{Cayrol2020,Lagasquie-Schiex2021a,Lagasquie-Schiex2021b,Lagasquie-Schiex2023}. In contrast, HQAF is encoded into Łukasiewicz three-valued logic and into continuous t-norm based fuzzy logics (Gödel, Product, Łukasiewicz). We prove equivalences between adjacent complete labellings, three-valued equational semantics, and the normal encoding, and relate fuzzy models to three-valued models via ternarization. This provides a lighter-weight logical foundation than first-order encodings.

\paragraph{Gradual approaches}
Gradual semantics for bipolar argumentation assign numerical strengths or rankings to arguments and interactions, providing a finer-grained view than extension-based semantics. Amgoud and Ben-Naim \cite{Amgoud2018,amgoudWeightedBipolarArgumentation2018} provide an axiomatic foundation for evaluating arguments in weighted bipolar graphs, where arguments have basic strengths and may be both supported and attacked; they propose a novel gradual semantics for acyclic graphs that satisfies all their axioms. Yun and Vesic \cite{Yun2021} define gradual semantics for weighted bipolar SETAF, where attacks and supports carry weights and arguments receive overall strength values. Gonzalez et al. \cite{Gonzalez2021} propose labeled bipolar argumentation frameworks, where labels represent additional argument features and are propagated through an algebra of labels. Wang and Shen \cite{Wang2026} propose convergent bilateral gradual semantics for weighted bipolar argumentation graphs, ensuring convergence even in cyclic cases. These approaches are valuable for ranking arguments and handling uncertainty, but they typically treat support as a single, undifferentiated relation: necessary, deductive, and evidential supports are not distinguished, and their different semantic constraints are not encoded separately. In contrast, the encoding-equational semantics developed in this paper explicitly distinguishes the three support types and assigns each of them a dedicated logical and equational treatment. In this sense, our work is not a replacement but a complement to gradual semantics: gradual semantics provide strength-based ranking, while HQAF provides a type-sensitive logical foundation for reasoning with multiple kinds of support and attack.

\paragraph{Comparison and novelty}
The gaps identified in the Introduction are addressed as follows.
\begin{itemize}
	\item \emph{Single vs.\ multiple support types:} existing bipolar frameworks consider one support interpretation \cite{Nouioua2011,Boella2010,Oren2008}; HQAF allows necessary, deductive, and evidential supports together with attacks.
	\item \emph{Higher-order but bipolar:} AFRA, ASAF, RAFN, and REBAF support recursive interactions but remain bipolar \cite{Baroni2011,Cohen2015,Cayrol2018a,Cayrol2018b}; HQAF is quadripolar and allows all four interaction types to be sources and targets.
	\item \emph{Encoding weight:} existing higher-order bipolar encodings are first-order or logic-programming based \cite{Cayrol2020,Lagasquie-Schiex2021a,Lagasquie-Schiex2021b,Lagasquie-Schiex2023,Alfano2024b}; HQAF is encoded into three-valued and fuzzy propositional logics, enabling lightweight solvers.
	\item \emph{Equivalence and cycles:} HQAF proves equivalence between adjacent complete labellings, three-valued equational semantics, and the normal encoding, and uses continuous fixed-point equations to guarantee solutions even with cycles.
\end{itemize}
In summary, HQAF does not replace existing bipolar or multipolar frameworks; it fills the gap of a uniform, higher-order, quadripolar setting in which three support types and attacks coexist, and it provides a lightweight logical encoding with proved semantic equivalences.

\section{Conclusion}
\label{sec:conclusion}

This paper has proposed a Higher-Order Quadripolar Argumentation Framework (HQAF) that unifies attacks, necessary supports, deductive supports, and evidential supports in a single setting. HQAF allows not only arguments but also attacks and supports to act as sources and targets of interactions, thereby providing a uniform treatment of higher-order interactions. We defined adjacent complete labelling semantics, 3-valued equational semantics, and encoded semantics for HQAF, and developed a normal encoding methodology that directly interprets HQAF in propositional logic systems: the 3-valued semantics is encoded into Łukasiewicz three-valued logic, while the fuzzy semantics is encoded into fuzzy propositional logics such as Gödel, Product, and Łukasiewicz logics. We proved model equivalence between HQAF and their encoded logical formulas, and investigated the relationships between 3-valued complete semantics and fuzzy encoded semantics via ternarization. In particular, the fuzzy equational semantics arises directly as an equivalent expression of the fuzzy encoded semantics, rather than being introduced separately.

A limitation of the present work is that we have focused on the non-set style of HQAF, where each interaction has a single source. This restriction simplifies the semantics and the encoding, but it leaves open the treatment of collective attacks and supports, i.e., interactions whose source is a set of arguments or interactions. As observed in the literature \cite{Polberg2014}, set-style supports are not semantically uniform: necessary set support is typically interpreted disjunctively (the acceptance of the target requires at least one member in each source set to be accepted), whereas evidential set support is interpreted conjunctively (the acceptance of the target requires each member in at least one source set to be accepted). It is also unclear how deductive set support should be interpreted in this respect. Moreover, for a higher-order quadripolar framework with collective interactions, it is difficult to transform set deductive supports into necessary supports simply by reversing the direction of support arrows. These issues motivate our decision to study the non-set style in this paper.

Future work will extend HQAF to the set style by developing a uniform semantics for collective attacks and for the three types of collective supports, and by investigating the relationships and possible translations among them. Another direction is to study other semantics, such as stable and preferred, and to explore the computational properties of the encoded theories. Finally, we plan to apply HQAF to real-world argumentation scenarios that require reasoning with multiple kinds of support and higher-order interactions.



\bibliographystyle{elsarticle-num} 
\bibliography{refHQabb}




\end{document}